\documentclass[12pt,reqno]{amsart}
\usepackage{amsfonts}
\usepackage{a4wide}

\numberwithin{equation}{section}

\newtheorem{theorem}{Theorem}[section]
\newtheorem{proposition}[theorem]{Proposition}
\newtheorem{corollary}[theorem]{Corollary}
\newtheorem{lemma}[theorem]{Lemma}

\theoremstyle{definition}

\newtheorem{definition}[theorem]{Definition}

\theoremstyle{remark}
\newtheorem{remark}[theorem]{Remark}

\newcommand{\be}{\begin{equation}}
\newcommand{\ee}{\end{equation}}

\begin{document}

\title[  threshold for fractional Hartree equation
 ]
{ Sharp Threshold of Blow-up and
  Scattering for the fractional Hartree equation }

 \author{Qing Guo and Shihui Zhu
}

\address{College of Science, Minzu University of China, Beijing 100081, China}
\email{guoqing0117@163.com}

\address{Department of Mathematics,
Sichuan Normal University\\ Chengdu, Sichuan 610066, China}
\email{ shihuizhumath@163.com}

%\thanks{This work is partially supported by ARC}
\begin{abstract} We consider the  fractional Hartree equation in the  $L^2$-supercritical  case, and we find a  sharp threshold of the scattering versus blow-up dichotomy for radial data: If $ M[u_{0}]^{\frac{s-s_c}{s_c}}E[u_{0}]<M[Q]^{\frac{s-s_c}{s_c}}E[Q]$ and $M[u_{0}]^{\frac{s-s_c}{s_c}}\|
u_{0}\|^2_{\dot H^s}<M[Q]^{\frac{s-s_c}{s_c}}\| Q\|^2_{\dot H^s}$, then the solution $u(t)$ is globally well-posed and scatters; if $ M[u_{0}]^{\frac{s-s_c}{s_c}}E[u_{0}]<M[Q]^{\frac{s-s_c}{s_c}}E[Q]$ and $M[u_{0}]^{\frac{s-s_c}{s_c}}\|
u_{0}\|^2_{\dot H^s}>M[Q]^{\frac{s-s_c}{s_c}}\| Q\|^2_{\dot H^s}$, the solution $u(t)$ blows up in finite time. This condition is sharp in the sense that the solitary wave  solution $e^{it}Q(x)$  is global but not scattering, which satisfies the equality in the above conditions. Here, $Q$ is the ground-state solution for the fractional Hartree equation.

\end{abstract}

\maketitle


\section{Introduction}

%Meltzler and Klafter  studied recent developments in the description of anomalous diffusion
%by the fractional dynamics approach, and  derived many fractional partial differential
%equations  asymptotically from L\'{e}vy random walk models, which is a natural generalization of
%the Brownian walk models(see\cite{MetzlerKlafter2000,MetzlerKlafter2002}).

In this paper, we   study   the   fractional Hartree
equation, which is the $L^2$-supercritical,  nonlinear, fractional
Schr\"{o}dinger equation.
\begin{equation}\label{eq1} iu_t-(-\triangle)^s
u+(\frac{1}{|x|^{\gamma}}*|u|^{2})u=0, \ \ \  \end{equation}
with the parameters $0<s<1$ and $2s<\gamma<\min\{N,4s\}$,
where  $i$ is the imaginary unit and $u=u(t,x)$: $\mathbb{R}\times \mathbb{R}^N \to \mathbb{C}$ is a complex valued function.  The operator $(-\triangle)^{s}$ is defined by
\[(-\triangle)^{s}u=\frac{1}{(2\pi)^{\frac N2}}\int  e^{ix\cdot\xi}|\xi|^{2s}\widehat{u}(\xi)d\xi=\mathcal{F}^{-1}[|\xi|^{2s}\mathcal{F}[u](\xi)],\]   where $\mathcal{F}$ and
$\mathcal{F}^{-1}$ are the Fourier transform and the Fourier inverse
transform in $\mathbb{R}^N$, respectively.
The fractional  Schr\"{o}dinger equations were first proposed by
Laskin  in \cite{Laskin2000,Laskin2002}  using the theory of functionals  over functional measures
generated  from the L\'{e}vy stochastic process
and by expanding the Feynman path
integral from the Brownian-like to  the L\'{e}vy-like quantum mechanical paths.  Here, $s$ is the L\'evy index.
In particular, if
$s=\frac12$ and $\gamma=1$, then \eqref{eq1}  models the dynamics of (pseudo-relativistic) boson stars, where the potential $\frac{1}{|x|}$ is the Newtonian gravitational potential in the appropriate physical units (see\cite{ElgartSchlein2007,Lenzmann2007}).
This equation is also called the pseudo-relativistic Hartree equation, whose  global existence and blow-up  have  been widely studied in \cite{FrohlichLenzmann2007,LenzmannLewin2010}.

Eq.(\ref{eq1}) is the   $L^2$-supercritical,  nonlinear, fractional
Schr\"{o}dinger equation.
Indeed, we remark on the scaling invariance of Eq.(\ref{eq1}). If $u(t,x)$ is  a solution of Eq.(\ref{eq1}), then $u^{\lambda}(t,x)=\lambda^{\frac{N-\gamma+2s}{2}}u(\lambda^{2s}t,\lambda x)$ is also a solution of Eq.(\ref{eq1}). This implies that
 \begin{itemize}
\item [(1)] $\|u^{\lambda}\|_{L^{p_c}}= \|u\|_{L^{p_c}}$, where $p_c=\frac{2N}{N-\gamma+2s}$. Moreover, when $\gamma>2s$, we see that $p_c>2$, and Eq. (\ref{eq1}) is called the $L^2$-supercritical, nonlinear, fractional  Schr\"{o}dinger equation.
\item [(2)] $\dot{H}^{s_c}$-norm is invariant for Eq. (\ref{eq1}), i.e., $\|u^{\lambda}\|_{\dot{H}^{s_c}}= \|u\|_{\dot{H}^{s_c}}$, where $s_c=\frac{\gamma-2s}{2}$.
\end{itemize}
Now, we impose  the initial data, \begin{equation}
 \label{1.2} u(0,x)=u_{0}\in H^s,
\end{equation}
onto (\ref{eq1}) and  consider the Cauchy problem  (\ref{eq1})-(\ref{1.2}). Cho et al  in \cite{ChoHwangKwonLee2012,ChoHwangHajaiejOzawa2012}
 established the local well-posedness  in $H^s$ as follows:
Let $N\geq 2$, $\frac12\leq s<1$ and $0<\gamma<\min\{N,4s\}$. If the
initial data $u_{0}\in H^s$, then
 there exists a unique solution $u(t,x)$
of the Cauchy problem (\ref{eq1})-(\ref{1.2}) on the maximal time
interval $I=[0,T)$ such that $u(t,x)\in C(I;H^{s})\bigcap
C^1(I;H^{-s})$ and either $T=+\infty$ (global existence) or
both $0<T<+\infty$ and $\lim\limits_{t\to T} \|
u(t,x)\|_{H^{s}} =+\infty$ (blow-up). Moreover, for
all $t\in I$, $u(t,x)$ satisfies the following conservation
laws.
\begin{itemize}
\item [(i)] Conservation of energy:
\begin{equation}\label{E}E[u(t)]= \frac 12 \int _{\mathbb{R}^N}\overline{u}(-\triangle )^{s}udx -\frac {1}{4} \int_{\mathbb{R}^N}\int_{\mathbb{R}^N}\frac{|u(x)|^2|u(y)|^2}{|x-y|^\gamma}  dxdy=E[u_0].\end{equation}
 \item [(ii)] Conservation of mass:
 \begin{equation}\label{M}M[u(t)]=\int_{\mathbb{R}^N} |u(t,x)|^2dx=M[u_0].\end{equation}
\end{itemize}
Now, even less is known about the global well-posedness and scattering results. To the authors' knowledge, Cho  et al in \cite{ChoHwangHajaiejOzawa2012}  gave
some small data  results.
First, they addressed the energy-supercritical case, i.e., $4s\leq\gamma<N$, and set some $\alpha>\frac{\gamma-2s}2$.
Assume that the initial data $\|u_0\|_{H^\alpha}$ are sufficiently small; then, there exists a unique solution
$u\in C_b([0,\infty);H^\alpha)\cap L^2(0,\infty;H^{\alpha+s-1}_{\frac{2N}{N-2}})$,
where $H^\alpha_q=(-\Delta)^{-\frac\alpha 2}L^q$. Moreover, there is $\phi^+\in H^\alpha$ such that
$$\|u(t)-e^{-i(-\Delta)^{s}}\phi^+\|_{H^\alpha}\rightarrow0,\ as\ t\rightarrow\infty.$$
Moreover, for the energy-subcritical case and for sufficiently small radial data $u_0\in H^s_{rad}$, they presented
some global well-posedness results: for $\frac N{2N-1}\leq s<1$, $2s<\gamma<\min\{4s,N\}$, there exists a
unique solution \[u(tax)\in C_b([0,\infty); H^s_{rad})\cap L_{loc}^{\frac{6s}{\gamma-2s}}(0,\infty;H^s_{\frac{2N}{N-\frac{2\gamma-4s}{3}}}).\]
However, they did not consider the scattering results in this case. On the other hand,  as a typical dispersive wave equation, under certain conditions, the solution of the nonlinear fractional Schr\"{o}dinger equation (\ref{eq1}) may blow-up in  finite time.
In light of the above phenomena,
a natural question would be
how small of initial data will induce the global existence of the solution. Furthermore, does this global solution scatter at either side of time?

Motivated by this problem, we study the scattering versus blow-up dichotomy of the solutions for the focusing $L^2$-supercritical,  nonlinear, fractional
Schr\"{o}dinger equation (\ref{eq1}).  Similar to studies on the classical semi-linear Schr\"{o}dinger equation (see\cite{Cazenave2003,MiaoZhang2008,Zhang2002}), we attempt to use the variational method to find a sharp threshold of blow-up and global existence of the solutions to  (\ref{eq1}).
The first topic is the ground-state solution of the equation
 \be\label{eqQ}(-\triangle
)^s Q+Q-(\frac{1}{|x|^{\gamma}}*|Q|^2)Q=0,\ \  \ \  Q\in H^s(\mathbb{R}^N).\ee
The existence of a non-trivial solution of Eq. (\ref{eqQ}) has been studied in \cite{Hajaiej2013,Zhu2016}, and the  stability of related standing waves has been obtained in \cite{CHHO2015,Guohuang2012,ZhangZhu2015}. In \cite{Zhu2016}, the second author of this paper obtained a sharp Gagliardo-Nirenberg inequality, which reveals the
  variational characteristic of the ground-state solutions for Eq. (\ref{eqQ}):
Let $N\geq 2$, $0<s<1$ and $0<\gamma<\min\{N,4s\}$. Then, for all $  v\in
H^s$,
\begin{equation}
\label{G-N}\int\int
\frac{|v(x)|^2|v(y)|^2}{|x-y|^\gamma}dxdy\leq C_{GN}
\left\|
v\right\|_2^{\frac{4s-\gamma}{s}}\left\|v\right\|_{\dot{H}^s}^{\frac{\gamma}{s}},\end{equation}
where $Q$ is a solution of (\ref{eqQ}),
\begin{equation}\label{CGN}
C_{GN}=\frac{4s}\gamma\frac1{\left\|
Q\right\|_2^{\frac{4s-\gamma}{s}}\left\|Q\right\|_{\dot{H}^s}^{\frac{\gamma-2s}{s}}}
=\left(\frac{4s-\gamma}{\gamma}\right)^{\frac{\gamma}{2s}}\frac{4s}{(4s-\gamma)\|Q\|_2^2}.
\end{equation}
Given the fractional operator $(-\triangle )^s$, the classical Virial identity argument fails, and the
the existence of blow-up solutions for  (\ref{eq1}) presents a particular difficulty.
The numerical observations of blow-up solutions have been studied in \cite{BaoCaiWang2010,BaoDong2011}.
 The theoretical  proof of the existence of the blow-up solutions   of  (\ref{eq1})
has been presented by Cho et al in \cite{ChoHwangKwonLee2012}. They
   proved that if $\gamma= 2 s\geq 1$ and the initial energy is negative, then the life span $[0,T)$ of  the corresponding solutions must be finite (i.e., $T<+\infty$). In \cite{Zhu2016}, by establishing some new estimates, Zhu proved the existence of a finite-time blow-up solution for Eq. (\ref{eq1}) with $\gamma= 2 s$ and the dynamics of the blow-up solutions.
   We note that the sharp threshold of the blow-up solutions and global existence for Eq. (\ref{eq1}) with $\gamma> 2 s$ remains unknown.

In the present paper,  we first construct two invariant flows by injecting the sharp
Gagliardo-Nirenberg inequality proposed by Zhu in \cite{Zhu2016}, which strongly depend on the scaling index $s_c=\frac{\gamma-2s}{2}$.
Then, we obtain the sharp criteria of blow-up and scattering for the $L^2$-supercritical, nonlinear, fractional Schr\"{o}dinger Eq. (\ref{eq1}) in terms
of the arguments in \cite{qgcpde16,radial,km}. The main theorem is as follows.
\newpage
%后面概述散射的结果
\begin{theorem}\label{th1}
Let $N\geq 2$ and $2s<\gamma<\min\{N,4s\}$. Assume that $u_0\in H^s$ is radial and $ M[u_{0}]^{\frac{s-s_c}{s_c}}E[u_{0}]<M[Q]^{\frac{s-s_c}{s_c}}E[Q],
$ where $Q$ is the ground-state solution of \eqref{eqQ}.
\begin{itemize}
  \item [(i)] If $\frac N{2N-1}\leq s<1$ and
  \begin{equation*}M[u_{0}]^{\frac{s-s_c}{s_c}}\|
u_{0}\|^2_{\dot H^s}<M[Q]^{\frac{s-s_c}{s_c}}\| Q\|^2_{\dot H^s},\end{equation*}
  then the corresponding solution $u(t,x)$ of
(\ref{eq1})-(\ref{1.2}) exists globally in $H^s$. Moreover, $u=u(t)$ scatters in $H^s$. Specifically, there
exists $\phi_\pm\in H^s$ such that
$\lim\limits_{t\rightarrow\pm\infty}\|u(t)-e^{-it(-\Delta)^{s}}\phi_\pm\|_{H^s}=0$.
  \item [(ii)] Further, if  the initial data $u_0\in
H^{s_0}$ with $s_0=\max\{2s, \frac{\gamma+1}{2}\}$ and
  \begin{equation*}M[u_{0}]^{\frac{s-s_c}{s_c}}\|
u_{0}\|^2_{\dot H^s}>M[Q]^{\frac{s-s_c}{s_c}}\| Q\|^2_{\dot H^s}\end{equation*}
  satisfies $ |x| u_0\in L^2$  and $x\cdot \nabla u_0\in L^2$,  then the solution $u(t,x)$ of
  (\ref{eq1})-(\ref{1.2}) must blow up in finite time $0<T<+\infty$.
\end{itemize}

\end{theorem}

This paper is organized as follows. In Section 2, using   the Strichartz estimates, we establish
the small data theory and the long-time perturbation theory. We review properties
of the ground state  Q in Section 3 in connection with the sharp Gagliardo-Nirenberg estimate.  We can construct the invariant flows generated by the Cauchy problem of (\ref{eq1}) and (\ref{1.2}) and prove Theorem \ref{th1}  for the blow-up part (ii). In Section 4, we introduce the local virial identity
and prove Theorem \ref{th1}, except for the scattering claim in part (i).  By assuming that the threshold for scattering
is strictly below the threshold claimed, we construct a ``critical element", $u_c$,
that stands exactly at the boundary between scattering and non-scattering. This is
done through a profile decomposition lemma in $H^{s}$. We then show that time slices of $u_c(t)$, as a collection of functions
in $H^s$, form a precompact set in $H^s$ (and thus, $u_c$ has something in common with the
soliton $Q(x)$). This enables us to prove that $u_c$ remains localized uniformly in time. In
Section 5, by using the localization in Section 4, we deduce a contradiction with the conservation of mass at large times.

We conclude this section by  introducing  some notations.
$L^q:=L^q(\mathbb{R}^N)$,
$\|\cdot\|_q:=\|\cdot\|_{L^q(\mathbb{R}^N)}$,
the time-space mixed norm $$\|u\|_{L^qX}:=\left(\int_\mathbb R\|u(t,\cdot)\|_{X}^q\right)^{\frac1q},$$
$H^s:=H^s(\mathbb{R}^N)$, $\dot{H}^s:=\dot{H}^s(\mathbb{R}^N)$, and
$\int \cdot dx:=\int_{\mathbb{R}^N}\cdot dx$.  $\mathcal{F}v=\widehat{v}$ denotes  the Fourier
transform of $v$, which for $v\in L^1( \mathbb{R}^N)$ is given by
$\mathcal{F}v= \widehat{v} (\xi):= \int e^{-i
x\cdot\xi}v(x)dx$  for all $\xi\in \mathbb{R}^N$, and $\mathcal{F}^{-1}v$ is
 the inverse Fourier transform  of
$v(\xi)$. $\Re z$ and $\Im z$ are the
real and imaginary parts of the complex number $z$, respectively.
 $\overline{z}$ denotes the complex conjugate of the complex number $z$. The various positive
constants will  be denoted by $C$ or $c$.

\section{Local theory and Strichartz estimate }
In this paper, we study the Cauchy
problem \eqref{eq1}-\eqref{1.2} in the form of the following integral equation:
$$u(t) = U(t)u_0+i\int_0^tU(t-t^1)(\frac{1}{|x|^{\gamma}}*|u|^{2})u(t^1)dt^1
$$
where$$
U(t)\phi(x) = e^{-i(-\triangle)^st}\phi(x)=\frac{1}{(2\pi)^{\frac N2}}\int  e^{i(x\cdot\xi-|\xi|^{2s})}\widehat{\phi}(\xi)d\xi.$$
In this section, we first recall the local theory for Eq.~\eqref{eq1}~ by the radial Strichartz estimate
~(see \cite{G-W,KT}).
\begin{definition}\label{defadmissible}
For the given $\theta\in[0,s)$, we state that the pair $(q,r)$ is $\theta$-level admissible,
denoted by $(q,r)\in\Lambda_\theta$,  if
\begin{align}\label{gap}
q,r\geq2,\ \
\frac{2s}{q}+\frac N{r}=\frac N2-\theta
\end{align}
and
 \begin{align}\label{range}
 \frac{4N+2}{2N-1}\leq q\leq\infty,\ \
\frac1{q}\leq\frac {2N-1}2(\frac12-\frac1{r}),\ \ \
or\ \  2\leq q<\frac{4N+2}{2N-1},
  \ \frac1{q}<\frac {2N-1}2(\frac12-\frac1{r}).
 \end{align}
Correspondingly, we denote the dual $\theta$-level admissible pair by $(q',r')\in\Lambda'_{\theta}$~if
~$(q,r)\in\Lambda_{-\theta}$~with~$(q',r')$~is the H\"older~ dual to~$(q,r).$~
\end{definition}

\begin{proposition}\label{prostrichartz} (see\cite{G-W})
Assume that $N\geq2$ and that $u_0,f$ are radial; then  for $q_j,r_j\geq2,j=1,2$,
\be
\|U(t)\phi\|_{L^{q_1}L^{r_1}}\leq C\|D^\theta\phi\|_2,
\ee
where $D^\theta=(-\triangle)^{\frac\theta2}$,
\be\label{inhomo}
\|\int_0^tU(t-t^1)f(t^1)dt^1\|_{L^{q_1}L^{r_1}}\leq C\|f\|_{L^{q'_2}L^{r'_2}},
\ee
in which $\theta\in\mathbb R$, the pairs $(q_j,r_j)$ satisfy the range conditions \eqref{range}
and the gap condition
$$\frac{2s}{q_1}+\frac N{r_1}=\frac N2-\theta,\ \ \ \frac{2s}{q_2}+\frac N{r_2}=\frac N2+\theta.$$
\end{proposition}

\begin{definition}\label{defnorm}

We  define the following Srichartz norm
%$$\|u\|_{S(\Lambda_0)}=\sup_{(q,r)L^{2}~admissible}\|u\|_{L^qL^r},$$and
$$\|u\|_{S(\Lambda_{s_c})}=\sup_{(q,r)\in\Lambda_{s_c}}\|u\|_{L^qL^r}.
%\footnote{By notation ~$\|\cdot\|_{S(\dot{H}^{s_{c}})}$~ in the sequel, we in fact add the restriction~$q\geq r$~
 %to the definition of  ~$(q,r)\dot{H}^{s_{c}}$~admissible.}
 $$ Let $(q',r')$~be the H\"older~ dual to~$(q,r),$~  and define the dual Strichartz norm
$$\|u\|_{S'(\Lambda_{-s_c})}=\inf_{(q',r')\in\Lambda'_{s_c}}\|u\|_{L^{q'}L^{r'}}=\inf_{(q,r)\in\Lambda_{-s_c}}\|u\|_{L^{q'}L^{r'}}.$$
\end{definition}

\begin{remark}
Notice that if $$s\in[\frac N{2N-1},1)\subset(\frac12,1),$$
the gap condition \eqref{gap} with $\theta=0$
right implies the range condition \eqref{range}, which further means that $\Lambda_0$ is nonempty.
That is we have a full set of 0-level admissible Strichartz estimates without loss of derivatives in radial case.
Moreover, denoting
\begin{align}\label{qc}q_c=r_c=\frac{2N+4s}{N+2s-\gamma},\end{align}
 we check that $(q_c,r_c)\in \Lambda_{s_c}\neq\emptyset$
 is an $s_c$-level admissible pair.

\end{remark}

By Proposition \ref{prostrichartz}, for $\phi,f$ radial, we then have that
$$\|U(t)\phi\|_{S(\Lambda_0)}\leq C\|\phi\|_2$$ and
$$\left\|\int_0^tU(t-t^1)f(\cdot,t^1)dt^1\right\|_{S(\Lambda_0)}\leq C\|f\|_{S'(\Lambda_0)}.$$
Together with Sobolev embedding, we obtain
$$\|U(t)\phi\|_{S(\Lambda_{s_c})}\leq c\|\phi\|_{\dot{H}^{s_c}},$$
$$\left\|\int_0^tU(t-t^1)f(\cdot,t^1)dt^1\right\|_{S(\Lambda_{s_c})}\leq C\|D^{s_c}f\|_{S'(\Lambda_{0})}$$
and
$$\left\|\int_0^tU(t-t^1)f(\cdot,t^1)dt^1\right\|_{S(\Lambda_{s_c})}\leq C\|f\|_{S'(\Lambda_{-s_c})}.$$

Next, we write ~$S(\Lambda_{\theta};I)$~to indicate  its restriction to a time subinterval~$I\subset(-\infty,+\infty).$

\begin{proposition}\label{sd}
(Small data)
Let $\|u_{0}\|_{\dot H^{s_c}}\leq A$ be radial. Then,
there exists  $\delta_{sd}=\delta_{sd}(A)>0$ such that if
~$\|U(t)u_{0}\|_{S(\Lambda_{s_c})}\leq \delta_{sd}, $~then~$ u $~solving \eqref{eq1} is global, and
\begin{eqnarray}
&&\|u\|_{S(\Lambda_{s_c})}\leq
2\|U(t)u_{0}\|_{S(\Lambda_{s_c})},\\
&&\|D^{s_c}u\|_{S(\Lambda_0)}\leq
2c\|u_{0}\|_{\dot H^{s_c}}.
\end{eqnarray}
(Note that by
 the Strichartz estimates, the hypotheses are satisfied if
~$\|u_{0}\|_{\dot H^{s_c}}\leq C\delta_{sd}. $)
\end{proposition}

\begin{proof}
Set
$$\Phi_{u_0}(v)=U(t)u_{0}+i\int_0^tU(t-t^1)(\frac1{|\cdot|^\gamma}\ast|v|^2)v(t^1)dt^1.$$
By the Strichartz estimates, we have
$$\|D^{s_c}\Phi_{u_0}(v)\|_{S(\Lambda_0)}\leq c\|u_{0}\|_{\dot{H}^{s_c}}
+c\|D^{s_c}[(\frac1{|\cdot|^\gamma}\ast|v|^2)v]\|_{L^{q^{'}}L^{r^{'}}}$$
and
$$\|\Phi_{u_0}(v)\|_{S(\Lambda_{s_c})}\leq \|U(t)u_{0}\|_{S(\Lambda_{s_c})}
+c\|D^{s_c}[(\frac1{|\cdot|^\gamma}\ast|v|^2)v]\|_{L^{q^{'}}L^{r^{'}}},$$
with~$(q',r')\in\Lambda'_{0}.$~
Applying the fractional Leibnitz~\cite{ChoHwangHajaiejOzawa2012,kato,kato1995}~, the H\"older inequalities and the Hardy-Littlewood-Sobolev inequalities, we have
\begin{align*}
&\|D^{s_c}[(\frac1{|\cdot|^\gamma}\ast|v|^2)v]\|_{L^{q^{'}}L^{r^{'}}}
\leq c\|D^{s_c}v(\frac1{|\cdot|^\gamma}\ast|v|^2)\|_{L^{q^{'}}L^{r^{'}}}+ c\|[\frac1{|\cdot|^\gamma}\ast Re(\bar{v}D^{s_c}v)]v\|_{L^{q^{'}}L^{r^{'}}}\\
&\leq c\|D^{s_c}v\|_{L^{q_{1}}L^{r_{1}}}\|v\|^2_{L^{q_{2}}L^{r_{2}}}
+c\|v\|_{L^{\gamma_{1}}L^{\rho_{1}}}\|\bar{v}D^{s_c}v\|_{L^{\frac{\gamma}{2}}L^{\frac{\rho}{2}}}\\
&\leq c\|D^{s_c}v\|_{L^{q_{1}}L^{r_{1}}}\|v\|^2_{L^{q_{2}}L^{r_{2}}}
+c\|v\|_{L^{\gamma_{1}}L^{\rho_{1}}}\|v\|_{L^{\gamma_{2}}L^{\rho_{2}}}
\|D^{s_c}v\|_{L^{\gamma_{3}}L^{\rho_{3}}}\\
&\leq c\|v\|^2_{S(\Lambda_{s_c})}\|D^{s_c}v\|_{S(\Lambda_0)},
\end{align*}
where the pairs $(q,r),(q_1,r_1)\in\Lambda_{0},$
$(q_2,r_2),(\gamma_1,\rho_1),(\gamma_2,\rho_2)\in\Lambda_{s_c}$, which indeed can be chosen as
$(q_2,r_2)=(\gamma_1,\rho_1)=(\gamma_2,\rho_2)=(q_c,r_c)\in\Lambda_{s_c}$.
Let $$\delta_{sd}\leq\min\left(\frac{1}{\sqrt{8}c},\frac{1}{8c^3A}\right),$$
%\min\left(\frac{1}{\sqrt{24}c},\frac{1}{24cA}\right),$$
and $$B=\left\{v|\|v\|_{S(\Lambda_{s_c})}\leq 2\|U(t)u_{0}\|_{S(\Lambda_{s_c})},
\|D^{s_c}v\|_{S(\Lambda_0)}\leq
2c\|u_{0}\|_{\dot{H}^{s_c}}
\right\}.$$
Then, ~$\Phi_{u_0}:B\rightarrow B$~and is a contraction on~$B$~; thus, the fixed point principle gives the result.

\end{proof}

\begin{proposition}\label{h1scattering}
If $u_{0}\in H^s$ is radial and $u(t)$ is global with  both bounded $s_c$-level Strichartz norm $\|u\|_{S(\Lambda_{s_c})}<\infty$ and
uniformly bounded $H^s$ norm
$\sup\limits_{t\in[0,\infty)}\|u\|_{H^s}\leq B,$
then $u(t)$ scatters in $H^s$ as $t\rightarrow +\infty.$
Specifically, there exists $\phi^{+}\in H^s$ such that
$$\lim_{t\rightarrow+\infty}\|u(t)-U(t)\phi^{+}\|_{H^s}=0.$$

\end{proposition}

\begin{proof}
We can obtain from the integral equation
\begin{align}\label{203}
u(t)=U(t)u_{0}+i\int_0^tU(t-t^1)(\frac1{|\cdot|^\gamma}\ast|u|^2)u(t^1)dt^1
\end{align}
that
\begin{align}\label{2.03}
u(t)-U(t)\phi^{+}=-i\int_t^{\infty}U(t-t^1)(\frac1{|\cdot|^\gamma}\ast|u|^2)u(t^1)dt^1,
\end{align}
where $\phi^{+}=u_{0}+i\int_0^{\infty}U(-t^1)(\frac1{|\cdot|^\gamma}\ast|u|^2)u(t^1)dt^1$.
By the Hardy-Littlewood-Sobolev inequality and the Strichartz estimates, for $0\leq\alpha\leq s$,  
there exist some $(q,r)\in\Lambda_0$, $(q_1,r_1)\in\Lambda'_{0}$ such that
\begin{align}\label{204}
\left\|D^\alpha\left(\int_{I}U(t-s)\left((\frac1{|\cdot|^\gamma}\ast|u|^2)u(s,x)\right)ds\right)\right\|_{L^{q}_{I}L^{r}} \nonumber
&\leq C\left\|D^\alpha\left((\frac1{|\cdot|^\gamma}\ast|u|^2)u\right)\right\|_{L^{q_1}_{I}L^{r_1}}\\ \nonumber
&\leq C\|D^\alpha u\|_{L_I^q L^r}\|\frac1{|\cdot|^\gamma}\ast|u|^2\|_{L^{q_2}_{I}L^{r_2}}\\ 
&\leq C\|D^\alpha u\|_{L_I^q L^r}\|u\|_{L^{q_c}_{I}L^{r_c}}^2,
\end{align}
where $I\subset[0,+\infty)$, $$\frac1{q_1}=\frac1{q_2}+\frac1q=\frac{2}{q_c}+\frac1q,\ \ \frac1{r_1}=\frac1r+\frac1{r_2}=\frac1r+
\frac\gamma N+\frac{2}{r_c}-1.$$
Since  $\|u\|_{L^{q_c}_{[0,\infty)}L^{r_c}}<\infty$, we can partition $[0,+\infty)$ into a union of $I_j=[t_j,t_{j+1}],1\leq j\leq N$, such that
for every $1\leq j\leq N$,  $\|u\|_{L^{q_c}_{I_j}L^{r_c}}<\delta (\delta$ is sufficiently small). Thus, by \eqref{203} and \eqref{204},
for $0\leq\alpha\leq s, \forall 1\leq j\leq N$,
\begin{align*}
\|D^\alpha u\|_{L^{q}_{I_j}L^{r}}
&\leq
\|U(t)u(t_j)\|_{L^{q}_{I_j}L^{r}}
+
\left\|D^\alpha\left(\int_{I_j}U(t-s)\left((\frac1{|\cdot|^\gamma}\ast|u|^2)u(s,x)\right)ds\right)\right\|_{L^{q}_{I_j}L^{r}}\\
&\leq \|U(t)u(t_j)\|_{L^{q}_{I_j}L^{r}}+ C\|D^\alpha u\|_{L^q_{I_j} L^r}\|u\|_{L^{q_c}_{I_j}L^{r_c}}^2\\
&\leq
CB+C\delta^2\|D^\alpha u\|_{L^q_{I_j} L^r}.
\end{align*}
By choosing $\delta$ such that $C\delta^2<\frac12$, we see that 
$\|D^\alpha u\|_{L^{q}_{I_j}L^{r}}<\infty,\ 1\leq j\leq N.$
So we have $$\|D^\alpha u\|_{L^{q}L^{r}}<\infty.$$
By \eqref{2.03}, we have for  $0\leq\alpha\leq s$,
\begin{align*}
\|D^\alpha (u(t)-U(t)\phi^{+})\|_{2}\leq \|u\|_{L^{q_c}_{[t,\infty)}L^{r_c}}^2\|D^\alpha u\|_{L^{q}_{[t,\infty)}L^{r}}.
\end{align*}
%$$\frac1{q_1}=\frac1{q_2}=\frac{2}{q_c}=1-\frac\gamma{N+2s},\ \ \frac1{r_1}=\frac\gamma N+\frac{2}{r_c}-\frac12=\frac12+\frac\gamma N-\frac{\gamma}{N+2s}.$$
Taking $\alpha=0$, $\alpha=s$ in the above inequality and
sending  $t\rightarrow+\infty$,  we obtain
 the claim.

\end{proof}

\begin{proposition}\label{properturb}
(Long-time perturbation theory) For any given $A\gg 1$, there exist $\epsilon_0=\epsilon_0(A)\ll 1$ and
$c=c(A)$ such that the following holds: Let $u=u(t,x)\in H^s$ be radial and solve (\ref{eq1})
for all $t$. Let $\tilde{u}=\tilde{u}(t,x)\in H^s$ for all $t$, and set
$$e\equiv i\tilde{u}_{t}-(-\Delta)^s \tilde{u}+(\frac1{|\cdot|^\gamma}\ast|\tilde{u}|^{2})\tilde{u}.$$
If $$\|\tilde{u}\|_{S(\Lambda_{s_c})}\leq A,\ \   \|e\|_{S'(\Lambda_{-s_c})}\leq \epsilon_0\ \
{\rm and}\ \ \|U(t-t_0)(u(t_0)-\tilde{u}(t_0)\|_{S(\Lambda_{s_c})}\leq \epsilon_0,$$
then$$\|u\|_{S(\Lambda_{s_c})}\leq c=c(A)<\infty.$$
\end{proposition}

\begin{proof}
Define $w=u-\tilde{u}$. Then, $w$ solves the equation
$$iw_t-(-\Delta)^s w+(\frac1{|\cdot|^\gamma}\ast|w+\tilde{u}|^2)w+(\frac1{|\cdot|^\gamma}\ast|w+\tilde{u}|^2)\tilde{u}-(\frac1{|\cdot|^\gamma}\ast|\tilde{u}|^2)\tilde{u}+e=0.$$
Specifically,
\begin{equation}\label{2.5}\begin{array}{lll}\vspace{0.3cm}
&iw_t-(-\Delta)^s w+(\frac1{|\cdot|^\gamma}\ast|w|^2)w+(\frac1{|\cdot|^\gamma}\ast (\bar{w}\tilde{u}))w+(\frac1{|\cdot|^\gamma}\ast (w\bar{\tilde{u}}))w\\
&\qquad+(\frac1{|\cdot|^\gamma}\ast|w|^2)\tilde{u} +(\frac1{|\cdot|^\gamma}\ast|\tilde{u}|^2)w+(\frac1{|\cdot|^\gamma}\ast (\bar{w}\tilde{u}))\tilde{u}+(\frac1{|\cdot|^\gamma}\ast (w\bar{\tilde{u}}))\tilde{u}+e=0.
\end{array}\end{equation}
Because $\|\tilde{u}\|_{S(\Lambda_{s_c})}\leq A,$ we can partition $[t_0,\infty)$
into $N=N(A)$ intervals $I_j=[t_j,t_{j+1})$  such that for each $0\leq j\leq N-1$,
  $\|\tilde{u}\|_{S(\Lambda_{s_c};I_j)}<\delta$ with the sufficiently small $\delta$ to be specified later.
 The integral equation of \ref{2.5} with initial time $t_j$ is
\begin{align}\label{2.6}
w(t)=U(t-t_j)w(t_j)+i\int_{t_j}^tU(t-s)W(\cdot,s)ds,
\end{align}
where
\begin{align*}
W=&(\frac1{|\cdot|^\gamma}\ast|w|^2)w+(\frac1{|\cdot|^\gamma}\ast (\bar{w}\tilde{u}))w+(\frac1{|\cdot|^\gamma}\ast (w\bar{\tilde{u}}))w\\
&+(\frac1{|\cdot|^\gamma}\ast|w|^2)\tilde{u}+(\frac1{|\cdot|^\gamma}\ast|\tilde{u}|^2)w+(\frac1{|\cdot|^\gamma}\ast (\bar{w}\tilde{u}))\tilde{u}+(\frac1{|\cdot|^\gamma}\ast (w\bar{\tilde{u}}))\tilde{u}+e.
\end{align*}
Applying the inhomogeneous Strichartz estimate \eqref{inhomo} on $I_j$, we have for  $(q_1,r_1)\in\Lambda_{-s_c}$
\begin{align}\label{2.7}
\|w\|_{S(\Lambda_{s_c};I_j)}&\leq \|e^{i(t-t_j)\Delta}w(t_j)\|_{S(\Lambda_{s_c};I_j)}
+c\|(\frac1{|\cdot|^\gamma}\ast|w|^2)w\|_{L_{I_j}^{q_1'}L^{r_1'}}\\ \nonumber
&+c\|(\frac1{|\cdot|^\gamma}\ast (\bar{w}\tilde{u}))w\|_{L_{I_j}^{q_1'}L^{r_1'}}
+c\|(\frac1{|\cdot|^\gamma}\ast (w\bar{\tilde{u}}))w\|_{L_{I_j}^{q_1'}L^{r_1'}}
+c\|(\frac1{|\cdot|^\gamma}\ast|w|^2)\tilde{u}\|_{L_{I_j}^{q_1'}L^{r_1'}}\\ \nonumber
&+c\|(\frac1{|\cdot|^\gamma}\ast|\tilde{u}|^2)w\|_{L_{I_j}^{q_1'}L^{r_1'}}
+c\|(\frac1{|\cdot|^\gamma}\ast (\bar{w}\tilde{u}))\tilde{u}\|_{L_{I_j}^{q_1'}L^{r_1'}}
+c\|(\frac1{|\cdot|^\gamma}\ast (w\bar{\tilde{u}}))\tilde{u}\|_{L_{I_j}^{q_1'}L^{r_1'}}\\ \nonumber
&+\|e\|_{S'(\Lambda_{-s_c})}.
\end{align}
%$$\frac N2(\frac\gamma N-\frac\gamma{N+2s})\leq\frac\gamma{N+2s}\leq\frac {2N-1}2(\frac\gamma N-\frac\gamma{N+2s})$$
%is equivalent to $$N\geq\frac1{2-\frac1s}$$
Under the condition $\frac{N}{2N-1}\leq s<1$, we easily obtain that
any $(q_i,r_i),i=1,2$ solving
\begin{equation}\label{qiri}
\begin{cases}
\frac1{q_1'}=\frac2{q_c}+\frac1{q_2}=1-\frac\gamma{N+2s}+\frac1{q_2},\\
\frac1{r_1'}=\frac\gamma N+\frac2{r_c}+\frac1{r_2}-1=\frac\gamma N-\frac\gamma{N+2s}+\frac1{r_2}\\
\end{cases}
 \end{equation}
 should satisfy the range condition \eqref{range}.
Hence, for the above pair $(q_1,r_1)\in\Lambda_{-s_c}$,
we can find $(q_2,r_2)\in\Lambda_{s_c}$ and apply the
 Hardy-Littlewood-Sobolev inequality and H\"older inequalities
to find that
%\begin{align}\label{2.7'}
\[\|(\frac1{|\cdot|^\gamma}\ast|\tilde{u}|^2)w\|_{L_{I_j}^{q_1'}L^{r_1'}}%&\leq\|\frac1{|\cdot|^\gamma}\ast|\tilde{u}|^2\|_{L_{I_j}^{q_1}L^{r_1}}\|w\|_{L_{I_j}^{q_2}L^{r_2}}\\
\leq\|\tilde{u}\|_{L_{I_j}^{q_c}L^{r_c}}^2\|w\|_{L_{I_j}^{q_2}L^{r_2}} 
\leq
\|\tilde{u}\|_{S(\Lambda_{s_c};I_j)}^2\|w\|_{S(\Lambda_{s_c};I_j)}
\leq \delta^2\|w\|_{S(\Lambda_{s_c};I_j)},\]
%\end{align}
\begin{align*}
\|(\frac1{|\cdot|^\gamma}\ast|w|^2)\tilde{u}\|_{L_{I_j}^{q_1'}L^{r_1'}}
&\leq\|\tilde{u}\|_{L_{I_j}^{q_2}L^{r_2}}\|w\|_{L_{I_j}^{q_c}L^{r_c}}^2
&\leq
\|\tilde{u}\|_{S(\Lambda_{s_c};I_j)}\|w\|_{S(\Lambda_{s_c};I_j)}^2
\leq \delta\|w\|_{S(\Lambda_{s_c};I_j)}^2.
\end{align*}
Similarly, we have other terms estimated in the same way, and we substitute all the estimates in \eqref{2.7} to obtain
\begin{align}\label{2.8}
\|w\|_{S(\Lambda_{s_c};I_j)}\leq &\|U(t-t_j)w(t_j)\|_{S(\Lambda_{s_c};I_j)}
+c\delta^2\|w\|_{S(\Lambda_{s_c};I_j)}\\ \nonumber
&+c\delta\|w\|_{S(\Lambda_{s_c};I_j)}
+c\|w\|_{S(\Lambda_{s_c};I_j)}^3+c\|e\|_{S'(\dot{H}^{-s_c};I_j)}\\ \nonumber
\leq&\|U(t-t_j)w(t_j)\|_{S(\Lambda_{s_c};I_j)}
+c\delta^2\|w\|_{S(\Lambda_{s_c};I_j)}\\ \nonumber
&+c\delta\|w\|^2_{S(\Lambda_{s_c};I_j)}
+c\|w\|_{S(\Lambda_{s_c};I_j)}^3+c\epsilon_0.
\end{align}
Now, if  $\delta\leq\min(1,\frac{1}{2\sqrt{c}})$ and
\begin{align}\label{2.9}
\|U(t-t_j)w(t_j)\|_{S(\Lambda_{s_c};I_j)}+c\epsilon_0\leq\min(1,\frac{1}{8\sqrt{c}}),
\end{align}
we obtain
\begin{align}\label{2.10}
\|w\|_{S(\Lambda_{s_c};I_j)}\leq 2\|U(t-t_j)w(t_j)\|_{S(\Lambda_{s_c};I_j)}
+2c\epsilon_0.
\end{align}
Next, we take $t=t_{j+1}$ in \eqref{2.6} and apply $U(t-t_{j+1})$ to both sides. We then obtain
\begin{align}\label{2.11}
U(t-t_{j+1})w(t_{j+1})=U(t-t_{j})w(t_j)+i\int_{t_j}^{t_{j+1}}U(t-s)W(\cdot,s)ds.
\end{align}
Note that the Duhamel integral is confined to $I_j$.  Similar to \eqref{2.8}, we have the estimate
\begin{align*}
\|U(t-t_{j+1})w(t_{j+1})\|_{S(\Lambda_{s_c})}
\leq&\|e^{i(t-t_j)\Delta}w(t_j)\|_{S(\Lambda_{s_c})}
+c\delta^2\|w\|_{S(\Lambda_{s_c};I_j)}\\ \nonumber
&+c\delta\|w\|_{S(\Lambda_{s_c};I_j)}
+c\|w\|_{S(\Lambda_{s_c};I_j)}^3+c\epsilon_0.
\end{align*}
Then,  \eqref{2.9} and \eqref{2.10} imply
\begin{align*}
\|U(t-t_{j+1})w(t_{j+1})\|_{S(\Lambda_{s_c})}
\leq&2\|U(t-t_{j})w(t_j)\|_{S(\Lambda_{s_c})}
+2c\epsilon_0.
\end{align*}
Now, iterate the beginning with $j=0$, and we obtain
\begin{align*}
\|U(t-t_{j})w(t_{j})\|_{S(\Lambda_{s_c})}
\leq&2^j\|U(t-t_{0})w(t_0)\|_{S(\Lambda_{s_c})}
+(2^j-1)2c\epsilon_0\leq2^{j+2}c\epsilon_0.
\end{align*}
Because the second part of \eqref{2.9} is needed for each $I_j$, $0\leq j\leq N-1$, we require that
\begin{align}\label{2.12}
2^{N+2}c\epsilon_0\leq\min(1,\frac{1}{2\sqrt{6c}}).
\end{align}
 Recall that $\delta$ is an absolute constant
to satisfy \eqref{2.9};  the given $A$ determines the number of time intervals $N$. Then, by \eqref{2.12},
$\epsilon_0$ is determined by $N=N(A)$. Thus, the iteration completes our proof.

\end{proof}

\newpage

\section{Variational Characteristic and Invariant Sets}

In this section, we first recall some variational characteristic of the ground state for Eq. (\ref{eq1}) given in \cite{Zhu2016}. Then, we can construct the invariant flows generated by the Cauchy problem of (\ref{eq1}) and (\ref{1.2}). Finally, we give some refined estimates of the invariant set of the global solutions, which are crucial for proving that  the global solutions will be scattering.

\begin{lemma}  (see \cite{Zhu2016}) Let $N\geq 2$, $0<s<1$ and $0<\gamma<\min\{N,4s\}$.  Suppose that
 $Q$ is  the ground-state
solution  of   (\ref{eqQ}). Then, we have  the following
Pohozaev identities:
\begin{equation}\label{P1}
\int
\overline{Q}(-\triangle)^sQdx+
\int|Q|^2dx-\int\int\frac{|Q(x)|^2|Q(y)|^2}{|x-y|^{\gamma}}dxdy=0,\end{equation}
 \begin{equation}\label{P2}
 \frac{N-2s}{2}\int \overline{Q}(-\triangle)^sQdx+\frac N2\int|Q|^2dx-\frac{2N-\gamma}{4}
 \int\int\frac{|Q(x)|^2|Q(y)|^2}{|x-y|^{\gamma}}dxdy=0.
 \end{equation}

\end{lemma}

\begin{remark}\label{reQ}
  Let $Q$ be the ground-state solution of (\ref{eqQ}). In terms of the Pohozaev identities (\ref{P1}) and (\ref{P2}), we can obtain the following properties.
\begin{itemize}
\item [(i)]  \[\int\int\frac{|Q(x)|^2|Q(y)|^2}{|x-y|^{\gamma}}dxdy=\frac{4s}{\gamma}\|Q\|_{\dot{H}^s}^2=\frac{4s}{4s-\gamma}\|Q\|_2^2.\]
\item [(ii)] \[E[Q]=\frac 12 \int \overline{Q}(-\triangle)^s  Q dx
 -\frac {1}{4} \int\int\frac{|Q(x)|^2|Q(y)|^2}{|x-y|^{\gamma}}dxdy=\frac{\gamma-2s}{2(4s-\gamma)}\|Q\|_2^2.\]
\item [(iii)] \[E[Q]M[Q]^{\frac{s-s_c}{s_c}}=\frac{\gamma-2s}{2(4s-\gamma)}\|Q\|_2^{\frac{2s}{s_c}}.\]
\item [(iv)] \[\|Q\|_{\dot{H}^s}^{2}M[Q]^{\frac{s-s_c}{s_c}}=\frac{\gamma}{4s-\gamma}\|Q\|_2^{\frac{2s}{s_c}}.\]
\end{itemize}
The general fractional Laplacian was first proposed by Caffarelli and Silvestre in \cite{Caffarelli-Silvestre2007}, and many researchers have studied the related time-independent Schr\"{o}dinger  equations with the fractional Laplacian (see\cite{Chen-Li-Li2017,Felmer-Quaas-Tan2012,Frank-Lenzmann2013,g-h17na,g-h17jde,Liu-Ma2016}).

\end{remark}

For the Cauchy problem  (\ref{eq1})-(\ref{1.2}), we can construct the following
 two invariant evolution
flows by the sharp G-N inequality (\ref{G-N}) and the conservation laws. Let $u\in H^{s}\setminus\{0\}$, and define
\[K_1=\{\|  u \|_{\dot{H}^s}^{2}M[u]^{\frac{s-s_c}{s_c}}< \|Q\|_{\dot{H}^{s}}^{2}M[Q]^{\frac{s-s_c}{s_c}},\  E[u]M[u]^{\frac{s-s_c}{s_c}}<  E[Q]M[Q]^{\frac{s-s_c}{s_c}}\} \]and
\[K_2=\{\|  u \|_{\dot{H}^s}^{2}M[u]^{\frac{s-s_c}{s_c}}>\|Q\|_{\dot{H}^{s}}^{2}M[Q]^{\frac{s-s_c}{s_c}},\  E[u]M[u]^{\frac{s-s_c}{s_c}}<  E[Q]M[Q]^{\frac{s-s_c}{s_c}}\}.\]
\begin{proposition} \label{invariant set} Let $N\geq 2$ and $Q$ be the ground-state solution of
(\ref{eqQ}).  If $0<s<1$ and  $2s<\gamma<\min\{N,4s\}$, then
$K_1$ and $K_2$ are invariant manifolds of (\ref{eq1}).
\end{proposition}

\begin{proof}
Denote $$V(u):= \int\int\frac{|u(t,x)|^2|u(t,y)|^2}{|x-y|^\gamma}dxdy.$$
Multiplying the definition of energy by $M[u]^{\frac{s-s_c}{s_c}}$ and using \eqref{G-N},
we have
\begin{align*}
M[u]^{\frac{s-s_c}{s_c}}E[u]=&\frac12 \|u(t)\|^{\frac{2(s-s_c)}{s_c}}_{2}\|D^s u(t)\|^2_{2}
-\frac1{4}V(u)\|u\|_2^{\frac{2(s-s_c)}{s_c}}\\
\geq& \frac12 (\|u(t)\|^{\frac{s-s_c}{s_c}}_{2}\|D^s u(t)\|_{2})^2
-\frac{C_{GN}}{4}(\|u(t)\|^{\frac{s-s_c}{s_c}}_{2}\|D^s u(t)\|_{2})^{\frac{\gamma}{s}}.
\end{align*}
Define $f(y)=\frac12y^2-\frac1{4}C_{GN}y^{\frac{\gamma}{s}}$. Then,
$f'(y)=y\left(1-C_{GN}\frac{\gamma}{4s}y^{\frac{\gamma-2s}{s}}\right)$, and
thus, $f'(y)=0$ when $y_0=0$ and
$y_1=\|Q\|^{\frac{s-s_c}{s_c}}_{2}\|D^s Q\|_{2}$. The graph of
$f$ has a local minimum at $y_0$ and a local maximum at $y_1$. Remark \ref{reQ} implies that
$f_{max}=f(y_1)=M[Q]^{\frac{s-s_c}{s_c}}E[Q]$.
This combined with energy conservation gives
\begin{align}\label{1}
f(\|u(t)\|^{\frac{s-s_c}{s_c}}_{2}\|D^s u(t)\|_{2})\leq M[u(t)]^{\frac{s-s_c}{s_c}}E[u(t)]=M[u_0]^{\frac{s-s_c}{s_c}}E[u_0]
<f(y_1).
\end{align}
Next, we shall prove Proposition \ref{invariant set} in the following two cases:

{\bf Case I:} If the initial data $u_0\in K_1$, i.e., $\|u_{0}\|^{\frac{s-s_c}{s_c}}_{2}\|D^s
u_{0}\|_{2}<y_1$, then by  \eqref{1} and the continuity of $\|D^s
u(t)\|_2$ in $t$, we have
for all time ~$t\in \mathbb{R},$~
\begin{equation}\label{2.3'}
\|  u(t) \|_{\dot{H}^s}^{2}M[u(t)]^{\frac{s-s_c}{s_c}}< \|Q\|_{\dot{H}^{s}}^{2}M[Q]^{\frac{s-s_c}{s_c}}.
\end{equation}
Indeed, if (\ref{2.3'}) is not true, then
there exists $t_1\in I$ such that $\|u(t_1)\|^{\frac{s-s_c}{s_c}}_{2}\|D^s
u(t_1)\|_{2}\geq y_1$.
Because the corresponding solution $u(t,x)\in C(I;H^s)$ is
continuous with respect to $t$, there exists $0<t_0\leq t_1$ such
that $\|u(t_0)\|^{\frac{s-s_c}{s_c}}_{2}\|D^s
u(t_0)\|_{2}= y_1$. Thus, injecting the conservation of energy
$E[u(t_0)]=E[u_0]$ and $\|u(t_0)\|^{\frac{s-s_c}{s_c}}_{2}\|D^s
u(t_0)\|_{2}= y_1$ into  (\ref{1}), we deduce that
\[f(y_1)>M[u_0]^{\frac{s-s_c}{s_c}}E[u_0]=M[u(t_0)]^{\frac{s-s_c}{s_c}}E[u(t_0)]\geq f(\|u(t_0)\|^{\frac{s-s_c}{s_c}}_{2}\|D^s
u(t_0)\|_{2})=f(y_1).\]
This is a contradiction. Hence, (\ref{2.3'}) is true, which implies that $K_1$ is an invariant set.

{\bf Case II:} If the initial data $u_0\in K_2$, i.e., $\|u_{0}\|^{\frac{s-s_c}{s_c}}_{2}\|D^s u_{0}\|_{2}>y_1$, then by  \eqref{1}
and the continuity of $\|D^s u(t)\|_2$ in $t$, we have  for all time $t\in I$ that
\begin{equation}\label{2.5'}
\|  u(t) \|_{\dot{H}^s}^{2}M[u(t)]^{\frac{s-s_c}{s_c}}>\|Q\|_{\dot{H}^{s}}^{2}M[Q]^{\frac{s-s_c}{s_c}},\end{equation}
which implies that $K_2$ is an invariant set. The proof is similar to {\bf Case I}.
\end{proof}

\begin{remark}\label{remdelta}
From the argument above, we
  can refine this analysis to obtain the following.  If the
condition $\|  u_0 \|_{\dot{H}^s}^{2}M[u_0]^{\frac{s-s_c}{s_c}}< \|Q\|_{\dot{H}^{s}}^{2}M[Q]^{\frac{s-s_c}{s_c}}$ holds, then there exists $\delta>0$ such that
$M[u]^{\frac{s-s_c}{s_c}}E[u]<(1-\delta)M[Q]^{\frac{s-s_c}{s_c}}E[Q]$,
and thus, there exists $\delta_0=\delta_0(\delta)$ such that
$\|u(t)\|^{\frac{s-s_c}{s_c}}_{2}\|D^su(t)\|_{2}<
(1-\delta_0)\|Q\|^{\frac{s-s_c}{s_c}}_{2}\|D^s Q\|_{2}$, where $u=u(t)$ is the corresponding solution to Eq. (\ref{eq1}).
\end{remark}

\begin{theorem}\label{dichotomy}
(Global versus blow-up dichotomy) Let $u_{0}\in H^{s}$, and let  $I=(T_{-},T_{+})$  be the
maximal time interval of existence of ~$u(t)$~ solving ~\eqref{eq1}.
\begin{itemize}
\item [(i)] If  $u_0\in K_1$,
then ~$I=(-\infty,+\infty)$,~ i.e., the solution exists globally in
time.
\item [(i)] If  $u_0\in K_2\bigcap H^{s_0}$ is radial, $ |x| u_0\in L^2$  and $x\cdot \nabla u_0\in L^2$,
where $s_0=\max\{2s, \frac{\gamma+1}{2}\}$, then the corresponding solution $u(t,x)$ of (\ref{eq1}) must blow up in  a finite time $0<T<+\infty$.
\end{itemize}
\end{theorem}

\begin{proof} (i) By the invariance of $K_1$, we see that (\ref{2.3'}) is true.
 In
particular, the $H^s$-norm of the solution $u$ is bounded, which
proves the global existence of the solution in this case.

(ii)  Denote
$A:=\left(\left(\frac{\gamma}{4s-\gamma}\right)^{s_c}\frac{\|Q\|_2^{2s}}{M[u_0]^{s-s_c}}
\right)^{\frac{1}{2s_c}}$. 
Using the invariance of $K_2$, we have $\|
 u(t) \|_{\dot{H}^s}^2>A^2$ for all $t\in I$.  It follows from  \cite{ChoHwangKwonLee2012,Zhu2016} that
$|x|u(t)\in L^2$ and $x \cdot\nabla u(t)\in L^2$,
and
 for all $t\in I$ (the maximal time interval), $\int \overline{u} x (-\Delta)^{1-s}x udx$ is
non-negative and   \be\label{4.18} \int
\overline{u} x (-\Delta)^{1-s}x udx\leq
\int_0^t\int_0^t\left(2\gamma
E[u(\tau)]-(\gamma-2s)\|u(\tau)\|_{\dot{H}^s}^2\right)d\tau dt+Ct+C.\ee  Applying the fact that for all $t\in I$,
$E[u(t)]=E[u_0]<\frac{\gamma-2s}{2\gamma}A^2$  and $\|u(t) \|_{\dot{H}^s}^2>A^2$
 to (\ref{4.18}), we deduce  that  for all $t\in I$
\[\int
\overline{u} x (-\Delta)^{1-s}x udx<
\int_0^t\int_0^t\left(2\gamma\frac{\gamma-2s}{2\gamma}A^2
 -(\gamma-2s)A^2\right)d\tau dt+Ct+C.\]
Hence, there exists a constant $C_0>0$ such that for all $t\in I$\[\int \overline{u} x
(-\Delta)^{1-s}x udx\leq -C_0t^2+Ct+C.\]
For sufficiently lage $|t|$, the left-hand side is negative, while $\int \overline{u} x (-\Delta)^{1-s}x udx$ is
non-negative, which means that both $T_{-}$ and $T_{+}$ are finite.  Specifically,
the solution $u(t,x)$ of the Cauchy problem (\ref{eq1})-(\ref{1.2})  blows  up in  finite time.

\end{proof}

\begin{lemma}\label{lemlowerbound}
Let $u_0\in K_1$. Furthermore,
take $\delta>0$ such that
$M[u_0]^{\frac{s-s_c}{s_c}}E[u_0]<(1-\delta)M[Q]^{\frac{s-s_c}{s_c}}E[Q]$.
If $u$ is a solution to problem \eqref{eq1} with initial data $u_0$,
then there exists $C_\delta>0$ such that for all $t\in\mathbb R$,
\begin{align}\label{4.7}
\|D^s u\|_2^2-\frac{\gamma}{4s}V(u)\geq
C_\delta\|D^s u\|_2^2.
\end{align}
\end{lemma}
\begin{proof}
By Remark \ref{remdelta}, there exists
$\delta_0=\delta_0(\delta)>0$ such that for all $t\in\mathbb R$,

\begin{align}\label{low}
\|u(t)\|^{\frac{s-s_c}{s_c}}_{2}\|D^s u(t)\|_{2}<
(1-\delta_0)\|Q\|^{\frac{s-s_c}{s_c}}_{2}\|D^s Q\|_{2}.
\end{align} Let
$$h(t)=\frac1{\|Q\|^{\frac{2(s-s_c)}{s_c}}_{2}\|D^s Q\|^2_{2}}(
\|u(t)\|^{\frac{2(s-s_c)}{s_c}}_{2}\|D^s
u(t)\|^2_{2}-\frac{\gamma}{4s}V(u)\|u(t)\|^{\frac{2(s-s_c)}{s_c}}_{2})$$
and  $g(y)=y^2-y^{\frac{\gamma}s}$. By the Gagliardo-Nirenberg
estimate \eqref{G-N} with the sharp constant $C_{GN}$ \eqref{CGN}, we can obtain
$h(t)\geq g\left(\frac{\|u(t)\|^{\frac{s-s_c}{s_c}}_{2}\|D^s u(t)\|_{2}}{\|Q\|^{\frac{s-s_c}{s_c}}_{2}\|D^s Q\|_{2}}\right)$.
By \eqref{low}, we restrict our attention to $0\leq y\leq1-\delta_0$.
The elementary argument gives a constant $C_\delta$ such that $g(y)\geq
C_\delta y^2$ if $0\leq y\leq1-\delta_0$. This indeed implies \eqref{4.7}.
\end{proof}

\begin{lemma}\label{lemcompare}
(Comparability of  gradient and energy)
Let $u_0\in K_1$. Then,
$$\frac{\gamma-2s}{2\gamma}\|D^s u(t)\|_{2}^2\leq E[u(t)]\leq\frac{1}{2}\|D^su(t)\|_{2}^2.$$
\end{lemma}

\begin{proof}
The expression of $E[u(t)]$ gives the second inequality immediately. The first inequality is obtained from
$$\frac{1}{2}\|D^s u\|_{L^2}^2-\frac{1}{4}V(u)\geq\frac{1}{2}\|D^s  u\|_{L^2}^2
\left(1-\frac{2s}{\gamma}\left(\frac{\|D^s  u\|_{2}\|  u\|_{2}^{\frac{s-s_c}{s_c}}}{\|D^s  Q\|_{2}\|  Q\|_{2}^{\frac{s-s_c}{s_c}}}\right)
^{\frac{2s_c}s}\right)\geq\frac{\gamma-2s}{2\gamma}\|D^s u\|_{L^2}^2,$$
where we have used \eqref{G-N}, \eqref{CGN} and \eqref{2.3'}.

\end{proof}

To establish the scattering theory, we need the existence result of the wave operator $\Omega^+: \phi^+\mapsto v_0.$

\begin{proposition}\label{wave operator}
(Existence of wave operators) Suppose that $ \phi^+\in H^s$ and that
\begin{equation}\label{4.11}
\frac{1}{2}M[\phi^+]^{\frac{s-s_c}{s_c}}\|D^s\phi^+\|_2^2<M[Q]^{\frac{s-s_c}{s_c}}E[Q].
\end{equation}
Then, there exists $v_0\in H^s$ such that $v$ globally solves \eqref{eq1} with initial data $v_0$ satisfying
$$\|D^sv(t)\|_2\|v_0\|_2^{\frac{s-s_c}{s_c}}\leq\|D^sQ\|_2\|Q\|_2^{\frac{s-s_c}{s_c}},\ \
M[v]=\|\phi^+\|_2^2,\ \  E[v]=\frac{1}{2}\|D^s\phi^+\|_2^2,$$
and
$$\lim_{t\rightarrow+\infty}\|v(t)-U(t)\phi^+\|_{H^s}=0.$$
Moreover, if $\|U(t)\phi^+\|_{S(\Lambda_{s_c})}\leq\delta_{sd}$, where $\delta_{sd}$ is defined in Proposition \ref{sd}, then
$$\|v\|_{S(\Lambda_{s_c})}\leq 2\|U(t)\phi^+\|_{S(\Lambda_{s_c})},\ \
         \|D^{s_c}v\|_{S(\Lambda_0)}\leq 2c\|\phi^+\|_{\dot{H}^{s_c}}.$$
\end{proposition}
\begin{proof} In this paper, we always use $v(t):=FNLS(t)v_0$
to denote the solution $v(t)$ of Eq.\eqref{eq1} with the initial data $v(0)=v_0$.
First, similar to the proof of the small data scattering theory Proposition \ref{sd}, we can solve the integral equation
\begin{equation}\label{4.12}
v(t)=U(t)\phi^+-i\int_t^\infty U(t-t^1)(\frac1{|\cdot|^\gamma}\ast|v|^2)v(t^1)dt^1
\end{equation}
 for $t\geq T$ with $T$ large. In fact, there exists $T>>1$ such that
$\|U(t)\phi^+\|_{S(\Lambda_{s_c};[T,\infty))}\leq \delta_{sd}.$
Now, from \eqref{4.12}, we again obtain by the Strichartz  estimate and the Hardy-Littlewood-Sobolev inequality that
\begin{align*}
&\|D^s v(t)\|_{S(\Lambda_0;[T,\infty))}
\leq c\|D^s\phi^{+}\|_{L^2}+c\|D^s[(\frac1{|\cdot|^\gamma}\ast|v|^2)v]\|_{L_{[T,+\infty)}^{q'}L^{r'}}\\
&\leq c\|D^s\phi^{+}\|_{L^2}+c\|D^s v\|_{L_{[T,+\infty)}^{q_{1}}L^{r_{1}}}\|v\|^2_{L_{[T,+\infty)}^{q_{2}}L^{r_{2}}}
+c\|v\|_{L_{[T,+\infty)}^{\gamma_{1}}L^{\rho_{1}}}\|\bar{v}D^{s}v\|_{L_{[T,+\infty)}^{\frac{\gamma}{2}}L^{\frac{\rho}{2}}}\\
&\leq c\|D^s\phi^{+}\|_{L^2}+c\|D^{s}v\|_{L_{[T,+\infty)}^{q_{1}}L^{r_{1}}}\|v\|^2_{L_{[T,+\infty)}^{q_{2}}L^{r_{2}}}
+c\|v\|_{L_{[T,+\infty)}^{\gamma_{1}}L^{\rho_{1}}}\|v\|_{L_{[T,+\infty)}^{\gamma_{2}}L^{\rho_{2}}}
\|D^{s}v\|_{L_{[T,+\infty)}^{\gamma_{3}}L^{\rho_{3}}}\\
&\leq c\|D^s\phi^{+}\|_{L^2}+c\|v\|_{S(\Lambda_{s_c};[T,+\infty))}^2\|D^s v(t)\|_{S(\Lambda_0;[T,+\infty))},
\end{align*}
where  $(q,r),(q_1,r_1)\in\Lambda_{0},$
$(q_2,r_2),(\gamma_1,\rho_1),(\gamma_2,\rho_2)\in\Lambda_{s_c}$, which indeed can be chosen as
$(q_2,r_2)=(\gamma_1,\rho_1)=(\gamma_2,\rho_2)=(q_c,r_c)\in\Lambda_{s_c}$, with $(q_c,r_c)$ defined by \eqref{qc}.
Similarly,
\begin{align*}
\| v(t)\|_{S(\Lambda_0;[T,+\infty))}
\leq c\|\phi^{+}\|_{L^2}+c\|v\|_{S(\Lambda_{s_c};[T,+\infty))}^2\|v(t)\|_{S(\Lambda_0;[T,+\infty))}.
\end{align*}
Following  Proposition \ref{sd}, we obtain for sufficiently large $T$
$$\|v\|_{S(\Lambda_0;[T,+\infty))}+\|D^s v\|_{S(\Lambda_0;[T,+\infty))}<2c \|\phi^{+}\|_{H^s}.$$
Using a similar approach with $t>T$,  we obtain
$$\|v-U(t)\phi^+\|_{S(\Lambda_0;[T,+\infty))}+\|D^s( v-e^{it\Delta}
\phi^+)\|_{S(\Lambda_0;[T,+\infty))} \rightarrow 0
\ \  {\rm as} ~T\rightarrow \infty,$$
which implies $v(t)-U(t)\phi^+\rightarrow0 $ in $H^s,$ and thus, $M[v]=\|\phi^{+}\|_{2}^2.$
Because $U(t)\phi^+\rightarrow0$ in $L^p$ for any $p\in(2,\frac{2N}{N-2s}]$ as $t\rightarrow+\infty$,  by the Hardy-Littlewood-Sobolev inequality,
 $V(U(t)\phi^+)\rightarrow0$. This together with the fact that
 $\|D^s U(t)\phi^+\|_2$ is conserved implies
 $$E[v]=\lim_{t\rightarrow+\infty}(\frac{1}{2}\|D^s U(t)\phi^+\|_2^2
 -\frac{1}{4}V(U(t)\phi^+))=\frac{1}{2}\|D^s \phi^+\|_2^2.$$
Considering \eqref{4.11}, we immediately obtain $M[v]^{\frac{s-s_c}{s_c}}E[v]<E[Q]M[Q]^{\frac{s-s_c}{s_c}}.$
  Note that
\begin{align*}
&\lim_{t\rightarrow+\infty}\|v(t)\|_2^{\frac{2(s-s_c)}{s_c}}\|D^s v(t)\|_2^2
=\lim_{t\rightarrow+\infty}\|U(t)\phi^+\|_2^{\frac{2(s-s_c)}{s_c}}\|D^s U(t)\phi^+\|_2^2\\
=&\|\phi^+\|_2^{\frac{2(s-s_c)}{s_c}}\|D^s \phi^+\|_2^2\leq2E[Q]M[Q]^{\frac{s-s_c}{s_c}}
=\frac{\gamma-2s}{\gamma}\|Q\|_2^{\frac{2(s-s_c)}{s_c}}\|D^s Q\|_2^2,
\end{align*}
where we used \eqref{4.11} and Remark \ref{reQ} in the last two steps.
Thus, due to Theorem \ref{dichotomy}, we can evolve $v(t)$ from $T$ back to time $0$
and complete our proof.
\end{proof}

\section{Critical solution and compactness}

From this section, we   begin to prove the scattering part of Theorem \ref{th1}. Let $u(t)$ be the solution of \eqref{eq1}
such that the assumption of Theorem \ref{th1} holds. Then, we know from Theorem \ref{dichotomy} that
 $u(t)$ is globally well-posed. Thus,  combined with Proposition \ref{h1scattering}, our goal is to show that
\begin{equation}\label{Sbound}
\|u\|_{S(\Lambda_{s_c})}<\infty,
\end{equation}
which implies that the solution of  \eqref{eq1} is $H^s$ scattering.

{\bf We say that $SC(u_0)$ holds if \eqref{Sbound} is true for the solution $u=u(t)$ with the initial data $u_0$.}

We first claim that  there exists $\delta>0$ such that if $E[u_0]M[u_0]^{\frac{s-s_c}{s_c}}<\delta$ and
$\|u_{0}\|_{2}^{\frac{s-s_c}{s_c}}\|D^s u_{0}\|_{2}<\|Q\|_{2}^{\frac{s-s_c}{s_c}}\|D^s Q\|_{2},$
 then \eqref{Sbound} holds. Indeed,
if $$E[u_0]M[u_0]^{\frac{s-s_c}{s_c}}<\frac{s_c}{\gamma}\delta_{sd}^{\frac{2s}{s_c}},$$
where $\delta_{sd}$ is simply the  $C\delta_{sd}$ appearing in Proposition \ref{sd}, and\ \ \ 
$\|u_{0}\|_{2}^{\frac{s-s_c}{s_c}}\|D^s u_{0}\|_{2}<\|Q\|_{2}^{\frac{s-s_c}{s_c}}\|D^s Q\|_{2},$ we obtain from Lemma \ref{lemcompare} that
$$\|u_0\|_{\dot{H}^{s_c}}^2
\leq\|u_{0}\|_{2}^{\frac{2(s-s_c)}{s}}\|D^s u_{0}\|_{2}^{\frac{2s_c}{s}}
\leq\left( \frac\gamma{s_c}E[u_0]M[u_0]^{\frac{s-s_c}{s_c}}\right)^{\frac{s_c}{s}}
\leq\delta_{sd}^2,$$
which implies that  $SC(u_0)$ holds by the small data theory. The claim holds for $\delta=\frac{s_c}{\gamma}\delta_{sd}^{\frac{2s}{s_c}}$.
Now, for each $\delta$, we define the set $S_\delta$ to be the collection of all such initial data in $H^s$ :
$$S_\delta=\{u_0\in H^s:\ \  E[u_0]M[u_0]^{\frac{s-s_c}{s_c}}<\delta \ \
{\rm and} \ \  M[u_{0}]^{\frac{s-s_c}{s_c}}\|D^s u_{0}\|_{2}^2<M[Q]^{\frac{s-s_c}{s_c}}\|D^s Q\|_{2}^2\}.$$
We also define that $(ME)_c=\sup\{\delta:\ \ u_0\in S_\delta\Rightarrow SC(u_0)\ \  holds \}.$
If $(ME)_c=M[Q]^{\frac{s-s_c}{s_c}}E[Q]$, then we are done. Thus, we assume now that
\begin{equation}\label{me}
(ME)_c<M[Q]^{\frac{s-s_c}{s_c}}E[Q].
\end{equation}
Then, there exists a sequence of solutions $u_n$ to \eqref{eq1} with $H^s$ initial data $u_{n,0}$
(note from the beginning of the above section that  we can rescale them to satisfy $\|u_n\|_2=1$ )
such that $\|D^s u_{n,0}\|_{2}<\|Q\|_{2}^{\frac{s-s_c}{s_c}}\|D^s Q\|_{2}$  and $E[u_{n,0}]\downarrow (ME)_c$ as
$n\rightarrow \infty,$
and   $SC(u_0)$ does not hold for any $n$.

Our goal in this section is to show the existence of an $H^s$ solution $u_c$ to \eqref{eq1} with initial data
$u_{c,0}$ such that $\| u_{c,0}\|_{2}^{\frac{s-s_c}{s_c}}\|D^s u_{c,0}\|_{2}<\|Q\|_{2}^{\frac{s-s_c}{s_c}}\|D^s Q\|_{2}$
and $M[u_c]^{\frac{s-s_c}{s_c}}E[u_c]= (ME)_c$
for which $SC(u_{c,0})$ does not hold. Moreover, we will show that
$K=\{u_c(\cdot,t)|0\leq t<\infty\}$ is precompact in  $H^s$.
This will play an important role in the rigidity theorem in the next section, which will
ultimately lead to a contradiction.

Prior to fulfilling  our main task, we will first introduce a profile decomposition lemma that is highly similar to that in \cite{radial}, which is for the cubic Schr\"{o}dinger equation in the spirit  of Keraani \cite{keraani2001}.

\begin{lemma}\label{lpd}
(Profile expansion) Let $\phi_{n}(x)$ be a radial and uniformly bounded
sequence in $H^s$.
Then, for each M, there exists a subsequence of $\phi_{n}$, also
denoted by $\phi_{n}$, and \\
(1) for each $1\leq j\leq M$, there
exists a (fixed in n) profile $\psi^{j}(x)$ in $H^s$,\\
 (2) for each $1\leq j\leq M$, there exists a sequence (in n) of time
shifts $t_{n}^{j}$, \\
(3) there exists a
sequence (in n) of remainders $W_{n}^{M}(x)$ in $H^s$ such that
$$\phi_{n}(x)=\sum_{j=1}^{M}U(-t_{n}^{j})\psi^{j}(x)+W_{n}^{M}(x).$$
The time and space sequences have a pairwise divergence property,
i.e., for $1\leq j\neq k\leq M$, we have
\begin{equation}\label{divergence}
\lim_{n\rightarrow+\infty}
|t_{n}^{j}-t_{n}^{k}|=+\infty.
\end{equation}
The remainder sequence has the following asymptotic smallness
property:
\begin{equation}\label{remainder}
\lim_{M\rightarrow+\infty}[\lim_{n\rightarrow+\infty}\|U(t)W_{n}^{M}\|_{S(\Lambda_{s_c})}]=0.
\end{equation}
For fixed M and any $0\leq \alpha\leq s$, we have the asymptotic
Pythagorean expansion:
\begin{equation}\label{hsexpansion}
\|\phi_{n}\|_{\dot{H}^{\alpha}}^{2}=\sum_{j=1}^{M}\|\psi^{j}\|_{\dot{H}^{\alpha}}^{2}+\|W_{n}^{M}\|_{\dot{H}^{\alpha}}^{2}+o_{n}(1).
\end{equation}
\end{lemma}

\begin{remark}\label{reproof}
The proof of the linear profile decomposition could simply follow the proof in \cite{qgcpde16} without any significant changes.
Furthermore, from the proof, the vanishing property \eqref{remainder} could  be improved to
\begin{equation}\label{remainder'}
\lim_{M\rightarrow+\infty}[\lim_{n\rightarrow+\infty}\|U(t)W_{n}^{M}\|_{L^qL^r}]=0,\ \ \forall (q,r)\ {\rm satisfies}\ \eqref{gap}\
{\rm with}\ \theta=s_c,
\end{equation}
especially,
\begin{equation}\label{remainder''}
\lim_{M\rightarrow+\infty}[\lim_{n\rightarrow+\infty}\|U(t)W_{n}^{M}\|_{L^\infty L^{\frac{2N}{N-2s_c}}}]=0.
\end{equation}
\end{remark}

 \begin{lemma}\label{energy  expansion}(Energy Pythagorean expansion)
 In the situation of Lemma \ref{lpd}, we have
 \begin{equation}\label{epe}
E[\phi_{n}]=\sum_{j=1}^{M}E[U(-t_n^j)\psi^{j}]+E[W_{n}^{M}]+o_n(1).
\end{equation}
\end{lemma}

\begin{proof}
According to \eqref{hsexpansion}, it suffices to establish that for all $M\geq1$,
 \begin{equation}\label{hartree term expansion}
V(\phi_{n})=\sum_{j=1}^{M}V(U(-t_n^j)\psi^{j})+V(W_{n}^{M}).
\end{equation}
There are only two cases to consider.
Case 1. There exists some $j$ for which $t_n^j$ converges to a finite number, which without loss of generality,
we assume is 0. In this case, we will show that $\lim\limits_{n\rightarrow\infty}V(W_n^M)=0$  for $M>j$,
 $\lim\limits_{n\rightarrow\infty}V(U(-t_{n}^{k})\psi^{k})=0$ for all $k\neq j$, and
 $\lim\limits_{n\rightarrow\infty}V(\phi_n)=V(\psi^{j})$, which gives \eqref{hartree term expansion}.
 Case 2. For all $j$, $|t_n^j|\rightarrow\infty$. In this case, we will show that
 $\lim\limits_{n\rightarrow\infty}V(U(-t_{n}^{k})\psi^{k})=0$ for all $k$ and that
 $\lim\limits_{n\rightarrow\infty}V(\phi_n)=\lim\limits_{n\rightarrow\infty}V(W_n^M)$, which  gives
 \eqref{hartree term expansion} again.

 For Case 1, we infer from the proof of Lemma \ref{lpd} that $W_n^{j-1}\rightharpoonup\psi^j$.
 By the compactness of the embedding $H^s_{r}\hookrightarrow L^{p},\forall p\in(2,\frac{2N}{N-2s})$, it follows
 from that Hardy-Littlewood-Sobolev inequalities that
$V(W_n^{j-1})\rightarrow V(\psi^j)$. Let $k\neq j$. Then, we obtain from \eqref{divergence}
that $|t_n^k|\rightarrow\infty$. As argued in the proof of Lemma \ref{lpd}, from the Sobolev embedding
and the $L^p$ spacetime decay estimates (or the dispersive estimates; see \cite{g-p-w08}) of the linear flow, we find that
$V(U(-t_{n}^{k})\psi^{k})\rightarrow0$. %%%%%%%%%%%%%%%%%%%%%%%%%%%%%%%%%%%%%%%%%%%%%%%%%%%%%%%%%%%%%%%%%%%%%%%%%%前面注销的证明中有色散估计
Recalling that $$
W_n^{j-1}=\phi_n-U(-t_{n}^{1})\psi^{1}-\cdots -U(-t_{n}^{j-1})\psi^{j-1},$$
we conclude that $V(\phi_n)\rightarrow V(\psi^j)$. Because
$$W_n^M=
W_n^{j-1}-\psi^j-U(-t_{n}^{j+1})\psi^{j+1}-\cdots -U(-t_{n}^{M})\psi^{M},$$
we also conclude that $\lim\limits_{n\rightarrow\infty}V(W_n^M)\rightarrow0$  for $M>j$.

Case 2 follows similarly from the proof of Case 1.
\end{proof}

\begin{proposition}\label{critical}
(Existence of a critical solution)
There exists a global solution $u_c$ in $H^s$ with initial data $u_{c,0}$ such that
$\|u_{c,0}\|_2=1,$
$$E[u_c]=(ME)_c<M[Q]^{\frac{s-s_c}{s_c}}E[Q],\ \ \ \|D^s u_c\|_2^2<M[Q]^{\frac{s-s_c}{s_c}}\|D^s Q\|_2^2,\ \ for \ \ all\ \ 0\leq t<\infty,$$
and $$\| u_c\|_{S(\Lambda_{s_c})}=+\infty.$$
\end{proposition}

\begin{proof}
Recall that we have obtained the sequence $\|u_n\|_2=1$ described at the beginning of this section satisfying
 $\|D^s u_{n,0}\|_{2}^2<M[Q]^{\frac{s-s_c}{s_c}}\|D^s Q\|_{2}^2$  and $E[u_{n,0}]\downarrow (ME)_c$ as
$n\rightarrow \infty.$ Each $u_n$ is global and non-scattering $\| u_n\|_{S(\Lambda_{s_c})}=+\infty.$
We  apply Lemma \ref{lpd} to $u_{n,0}$, which is uniformly bounded in $H^s$, to obtain
\begin{align}\label{5.7}
u_{n,0}(x)=\sum_{j=1}^{M}U(-t_n^j)\psi^{j}(x)+W_{n}^{M}(x).
\end{align}
Then, by Lemma \ref{energy  expansion} (Energy Pythagorean expansion), we further have
 \begin{equation*}
\sum_{j=1}^{M}\lim_{n\rightarrow\infty}E[U(-t_n^j)\psi^{j}]+\lim_{n\rightarrow\infty}E[W_{n}^{M}]=\lim_{n\rightarrow\infty}E[u_{n,0}]=(ME)_c.
\end{equation*}
Also by the profile expansion, we have
 \begin{equation*}
\|D^s u_{n,0}\|_2^2=\sum_{j=1}^{M}\|D^s U(-t_n^j)\psi^{j}\|_2^2+\|D^s W_{n}^{M}\|_2^2+o_n(1),
\end{equation*}
and
 \begin{equation}\label{5.9}
1=\|  u_{n,0}\|_2^2=\sum_{j=1}^{M}\| \psi^{j}\|_2^2+\|  W_{n}^{M}\|_2^2+o_n(1).
\end{equation}
We know from the proof of Lemma \ref{lemcompare} that each energy is nonnegative, and thus,
\begin{align}\label{5.8}
\lim_{n\rightarrow\infty}E[U(-t_n^j)\psi^{j}]\leq(ME)_c.
\end{align}

{\bf Claim A: only one $\psi^j\neq0$.}

If more than one $\psi^{j}\neq0$, we will show  a contradiction in the following, and thus, the profile expansion
will be reduced  to the case in which  only one  profile is non-trivial.

For this, by  \eqref{5.9}, we must have $M[\psi^{j}]<1$ for each $j$, which together with \eqref {5.8},
implies that for sufficiently large $n$,
 \begin{equation*}
M[U(-t_n^j)\psi^{j}]^{\frac{s-s_c}{s_c}}E[U(-t_n^j)\psi^{j}]<(ME)_c.
\end{equation*}
For a given $j$, if $|t_n^j|\rightarrow+\infty$, we assume $t_n^j\rightarrow+\infty$
or $t_n^j\rightarrow-\infty$  up to a subsequence. In this case, by the proof of
Lemma \ref{energy  expansion}, we have $\lim\limits_{n\rightarrow+\infty}V(U(-t_{n}^{k})\psi^{k})=0,$
and thus,
$\frac{1}{2}\| \psi^{j}\|_2^{\frac{2(s-s_c)}{s_c}}\|D^s \psi^{j}\|_2^2=
\frac{1}{2}\|U(-t_n^j) \psi^{j}\|_2^{\frac{2(s-s_c)}{s_c}}\|D^s U(-t_n^j)\psi^{j}\|_2^2<(ME)_c$.
Then, we obtain from the existence of wave operators (Proposition \ref{wave operator}) that
there exists $\tilde{\psi}^{j}$ such that
$$\|FNLS(-t_{n}^{j})\tilde{\psi}^{j}-U(-t_n^j)\psi^{j}\|_{H^s}\rightarrow0,\ \ {\rm as}\ \ n\rightarrow+\infty$$
with
$$\| \tilde{\psi}^{j}\|_2^{\frac{s-s_c}{s_c}}\|D^s FNLS(t)\tilde{\psi}^{j}\|_2<
\| Q\|_2^{\frac{s-s_c}{s_c}}\|D^s Q\|_2$$
$$\| \tilde{\psi}^{j}\|_2=\| \psi^{j}\|_2,\ \ \ E[\tilde{\psi}^{j}]=\frac{1}{2}\|D^s \psi^{j}\|_2^2,$$
and thus,
$$M[\tilde{\psi}^{j}]^{\frac{s-s_c}{s_c}}E[\tilde{\psi}^{j}]<(ME)_c,\ \ \ \|FNLS(t) \tilde{\psi}^{j}\|_{S(\Lambda_{s_c})}<+\infty.$$

If, on the other hand, for the given $j$,  $t_{n}^{j}\rightarrow t'$ finite, then
by the continuity of the linear flow in $H^s$, we have
$$U(-t_n^j) \psi^{j}\rightarrow U(-t') \psi^{j} \ \ \ {\rm strongly \ \ in}\ \ H^s.$$
In this case, we set $\tilde{\psi}^{j}=FNLS(t')[U(-t') \psi^{j}]$ so that
 $FNLS(-t')\tilde{\psi}^{j}=U(-t') \psi^{j}$.

Above all,  in either case, we have a new profile $\tilde{\psi}^{j}$ for the given $\psi^{j}$  such that
$$\|FNLS(-t_{n}^{j})\tilde{\psi}^{j}-U(-t_n^j)\psi^{j}\|_{H^s}\rightarrow0,\ \ {\rm as}\ \ n\rightarrow+\infty.$$
As a result, we can replace $U(-t_n^j) \psi^{j}$ by $FNLS(-t_{n}^{j})\tilde{\psi}^{j}$ in
\eqref{5.7} and obtain
\begin{align*}\label{5.7}
u_{n,0}(x)=\sum_{j=1}^{M}FNLS(-t_{n}^{j})\tilde{\psi}^{j}(x)+\tilde{W}_{n}^{M}(x),
\end{align*}
where
\begin{equation*}\label{rem}
\lim_{M\rightarrow+\infty}[\lim_{n\rightarrow+\infty}\|U(t)\tilde{W}_{n}^{M}\|_{S(\Lambda_{s_c})}]=0.
\end{equation*}

To use the perturbation theory to obtain a contradiction, we set
$v^j(t)=FNLS(t)\tilde{\psi}^{j}$, $u_n(t)=FNLS(t)u_{n,0}$ and
$$\tilde{u}_n(t)=\sum_{j=1}^{M}v^j(t-t_{n}^{j}).$$
Then, we have
$$i\partial_t\tilde{u}_n-(-\Delta)^s \tilde{u}_n+(\frac1{|\cdot|^\gamma}\ast|\tilde{u}_n|^2)\tilde{u}_n=e_n,$$
where $$e_n=(\frac1{|\cdot|^\gamma}\ast|\tilde{u}_n|^2)\tilde{u}_n-\sum_{j=1}^{M}(\frac1{|\cdot|^\gamma}\ast|v^j(t-t_{n}^{j})|^2)v^j(t-t_{n}^{j}).$$
In the near future,  we will prove the following two claims to obtain the contradiction:
\begin{itemize}
\item  Claim 1 - There exists a large constant $A$ independent of $M$ such that the following holds:
For any $M$, there exists $n_0=n_0(M)$ such that for $n>n_0$, $\|\tilde{u}_n\|_{S(\Lambda_{s_c})}\leq A.$
\item  Claim 2 - For each $M$ and $\epsilon>0$, there exist $n_1=n_1(M,\epsilon)$ such that for $n>n_1$,
$\|e_n\|_{L^{q'_1}L^{r'_1}}\leq \epsilon$ for some pair $(q_1,r_1)\in\Lambda_{-s_c}$.
\end{itemize}
Note that if the two claims hold true,  because $\tilde{u}_n(0)-u_n(0)=\tilde{W}_{n}^{M},$ there exists $M_1=M_1(\epsilon)$
such that for each $M>M_1$, there exists $n_2=n_2(M)$ satisfying
$\|U(t)(\tilde{u}_n(0)-u_n(0))\|_{S(\Lambda_{s_c})}\leq \epsilon.$
Thus,  now by the long-time perturbation theory  Proposition \ref{properturb},  we have
for sufficiently large $n$ and $M$ that $\|u_n\|_{S(\Lambda_{s_c})}<+\infty,$
which is a contradiction, giving Claim A. Thus, it suffices to show the above claims.

Let $M_0$ be sufficiently large such that $\|U(t)\tilde{W}_{n}^{M_0}\|_{S(\Lambda_{s_c})}\leq \delta_{sd}.$
Thus, we know from the definition of $\tilde{W}_{n}^{M_0}$ that for each $j>M_0$, it holds that
$\|U(t)v^j(-t_{n}^{j})\|_{S(\Lambda_{s_c})}\leq \delta_{sd}.$
Similar to the small data scattering and Proposition \ref{wave operator}, we obtain
\begin{align}\label{5.10}
\|v^j(t-t_{n}^{j})\|_{S(\Lambda_{s_c})}\leq  2
\|U(t)v^j(-t_{n}^{j})\|_{S(\Lambda_{s_c})}\leq 2\delta_{sd},
\end{align}
and
\begin{align}\label{5.10'}
\|D^{s_c}v^j(t-t_{n}^{j})\|_{S(\Lambda_0)}\leq c
\|v^j(-t_{n}^{j})\|_{\dot{H}^{s_c}}\ \ \ {\rm for}\ \ j>M_0.
\end{align}
Recall that
$\|v^j(-t_{n}^{j})-U(-t_n^j)\psi^{j}\|_{\dot{H}^{s_c}}\rightarrow 0$
as $n\rightarrow+\infty$.
Then, \eqref{5.10'} implies for $n$ large and $j>M_0$ that
\begin{align}\label{5.10''}
\|D^{s_c}v^j(t-t_{n}^{j})\|_{S(\Lambda_0)}\leq c
\|U(-t_n^j)\psi^{j}\|_{\dot{H}^{s_c}}=c
\|\psi^{j}\|_{\dot{H}^{s_c}}.
\end{align}
Thus, by elementary calculation, we have that
\begin{align}\label{5.11}
\|\tilde{u}_n\|^{q_c}_{L^{q_c}L^{q_c}}
&=\sum_{j=1}^{M_0}\|v^j\|^{q_c}_{L^{q_c}L^{q_c}}
+\sum_{j=M_0+1}^{M}\|v^j\|^{q_c}_{L^{q_c}L^{q_c}}+crossterms\\ \nonumber
&\leq\sum_{j=1}^{M_0}\|v^j\|^{q_c}_{L^{q_c}L^{q_c}}
+c\sum_{j=M_0+1}^{M}\|\psi^{j}\|^{q_c}_{\dot{H}^{s_c}}+crossterms.
\end{align}
Note first that by \eqref{divergence},  the $crossterm$ can be made bounded by taking $n_0$ as sufficiently large.
On the other hand, by \eqref{5.7} and Lemma \ref{lpd},
\begin{align}\label{5.12}
\|u_{n,0}\|_{\dot{H}^{s_c}}^2
&=\sum_{j=1}^{M_0}\|\psi^j\|_{\dot{H}^{s_c}}^2
+\sum_{j=M_0+1}^{M}\|\psi^j\|_{\dot{H}^{s_c}}^2
+\|W_n^{M}\|_{\dot{H}^{s_c}}^2+o_n(1),
\end{align}
which shows that the quantity $\sum_{j=M_0+1}^{M}\|\psi^j\|_{\dot{H}^{s_c}}^{\frac{2(N+2s)}{N+2s-\gamma}}$ is bounded
independently of $M$.
Hence, \eqref{5.11} gives that $\|\tilde{u}_n\|_{L^{q_c}L^{q_c}}$
 is bounded independently of $M$ for $n>n_0$.
 A similar argument will show that $\|\tilde{u}_n\|_{L^{\infty}L^{\frac{2N}{N-2s_c}}}$
is also bounded
independently of $M$ provided that $n>n_0$ is sufficiently  large.
According to the definition of the Strichartz  norm introduced in section 2,
the boundness of
of $\|\tilde{u}_n\|_{S(\Lambda_{s_c})}$ can be obtained by
interpolation between the two exponents. Then, finally,  we have obtained that Claim 1 holds true.

Now, we turn to prove the second claim. We easily have the following expansion of $e_n$:
\begin{align*}
e_n=&\left(\frac1{|\cdot|^\gamma}\ast|\sum_{j=1}^{M}v^j(t-t_{n}^{j})|^2\right)\sum_{j=1}^{M}v^j(t-t_{n}^{j})
-\sum_{j=1}^{M}\left(\frac1{|\cdot|^\gamma}\ast|v^j(t-t_{n}^{j})|^2\right)v^j(t-t_{n}^{j})\\
=&\left(\frac1{|\cdot|^\gamma}\ast\left(|\sum_{j=1}^{M}v^j(t-t_{n}^{j})|^2-\sum_{j=1}^{M}|v^j(t-t_{n}^{j})|^2\right)\right)\sum_{j=1}^{M}v^j(t-t_{n}^{j})\\
&+\left(\frac1{|\cdot|^\gamma}\ast\sum_{j=1}^{M}|v^j(t-t_{n}^{j})|^2\right)\sum_{j=1}^{M}v^j(t-t_{n}^{j})
-\sum_{j=1}^{M}\left(\frac1{|\cdot|^\gamma}\ast|v^j(t-t_{n}^{j})|^2\right)v^j(t-t_{n}^{j})\\
=&\left(\frac1{|\cdot|^\gamma}\ast\left(|\sum_{j=1}^{M}v^j(t-t_{n}^{j})|^2-\sum_{j=1}^{M}|v^j(t-t_{n}^{j})|^2\right)\right)\sum_{j=1}^{M}v^j(t-t_{n}^{j})\\
&+\sum_{j=1}^{M}\left(\frac1{|\cdot|^\gamma}\ast|v^j(t-t_{n}^{j})|^2\right)\sum_{k\neq j}v^k(t-t_{n}^{k}).
\end{align*}
The focus now is on how to estimate the cross terms.
Assume first that  $j\neq k$ and $|t_{n}^{j}- t_{n}^{k}|\rightarrow+\infty$; then,
taking one of the cross terms for example, we have
\begin{align}\label{crossex}
\left\|(\frac1{|\cdot|^\gamma}\ast|v^j|^2)(t-t_{n}^{j})v^k(t-t_{n}^{k})\right\|_{L^{q'_1}L^{r'_1}}&=
\left\|(\frac1{|\cdot|^\gamma}\ast|v^j|^2)(t)v^k(t+t_{n}^{j}-t_{n}^{k})\right\|_{L^{q'_1}L^{r'_1}}.
\end{align}
Using a similar argument as in \eqref{2.7'},
 for the above pair $(q_1,r_1)\in\Lambda_{-s_c}$,
we can find $(q_2,r_2)\in\Lambda_{s_c}$ and apply
the Hardy-Littlewood-Sobolev inequality and H\"older inequalities
to obtain
\begin{align*}
\left\|(\frac1{|\cdot|^\gamma}\ast|v^j|^2)(t)v^k(t+t_{n}^{j}-t_{n}^{k})\right\|_{L^{q'_1}L^{r'_1}}\leq&
\|v^j\|^2_{L^{q_c}L^{r_c}}\|v^k\|_{L^{q_2}L^{r_2}}\\
\leq&
\|v^j\|_{S(\Lambda_{s_c};I_j)}^2\|v^k\|_{S(\Lambda_{s_c};I_j)}.
\end{align*}
If  $j\neq k$, by \eqref{divergence}, $|t_{n}^{j}-t_{n}^{k}|\rightarrow+\infty$, and then,
 we find that \eqref{crossex} goes to zero as $n\rightarrow\infty$.
Observe that all other cross terms will  have the same property through similar estimates,
and  we have proved Claim 2.

Claim 1 and Claim 2 imply Claim A. We have reduced the profile expansion to the case in which
$\psi^1\neq0$, and $\psi^j=0$ for all $j\geq2$. We now begin to show the existence of a critical solution.

By \eqref{5.9}, we have $M[\psi^1]\leq1,$
and by \eqref{5.8}, we have $\lim\limits_{n\rightarrow\infty}E[U(-t_{n}^{1})\psi^{1}]\leq(ME)_c$.
If $t_{n}^{1}$ converges and, without loss of generality, $t_{n}^{1}\rightarrow 0$
as $n\rightarrow+\infty,$ we take $\tilde{\psi}^1=\psi^1$, and then, we have
$\|FNLS(-t_{n}^{1})\tilde{\psi}^1-U(-t_{n}^{1})\psi^1\|_{H^s}\rightarrow0$ as $n\rightarrow+\infty.$
If, on the other hand,  $t_{n}^{1}\rightarrow +\infty,$ then by the proof of
Lemma \ref{energy  expansion}, we have again  $\lim\limits_{n\rightarrow+\infty}V(U(-t_{n}^{1})\psi^{1})=0,$
and thus, $$\frac{1}{2}\|D^s \psi^{1}\|_2^2=\lim_{n\rightarrow\infty}E[U(-t_{n}^{1})\psi^{1}]\leq(ME)_c.$$
Therefore, by Proposition \ref{wave operator}, there exist
$\tilde{\psi}^1$ such that $M[\tilde{\psi}^1]=M[\psi^1]\leq1,$
$E[\tilde{\psi}^1]=\frac{1}{2}\|D^s \psi^{1}\|_2^2\leq(ME)_c,$
and
$\|FNLS(-t_{n}^{1})\tilde{\psi}^1-U(-t_{n}^{1})\psi^1\|_{H^s}\rightarrow0$ as $n\rightarrow+\infty.$

In either case, if we set $\tilde{W}_n^M=W_n^M+(U(-t_{n}^{1})\psi^1-FNLS(-t_{n}^{1})\tilde{\psi}^1),$
then by the Strichartz  estimates, we have
$$\|U(t)\tilde{W}_{n}^{M}\|_{S(\Lambda_{s_c})}
\leq\|U(t)W_{n}^{M}\|_{S(\Lambda_{s_c})}+
c\|U(-t_{n}^{1})\psi^1-FNLS(-t_{n}^{1})\tilde{\psi}^1\|_{S(\Lambda_{s_c})},$$
and thus, $$\lim_{n\rightarrow+\infty}\|U(t)\tilde{W}_{n}^{M}\|_{S(\Lambda_{s_c})}
=\lim_{n\rightarrow+\infty}\|U(t)W_{n}^{M}\|_{S(\Lambda_{s_c})}.$$
Therefore, we have
$$u_{n,0}=FNLS(-t_{n}^{1})\tilde{\psi}^1)+\tilde{W}_{n}^{M}$$
with $M(\tilde{\psi}^1)\leq1,$  $E(\tilde{\psi}^1)\leq(ME)_c$
and $\lim\limits_{M\rightarrow+\infty}[\lim\limits_{n\rightarrow+\infty}\|U(t)\tilde{W}_{n}^{M}\|_{S(\Lambda_{s_c})}]=0.$
Let $u_c$ be the solution to \eqref{eq1} with initial data $u_{c,0}=\psi^1$. Now, if we claim
that $\|u_c\|_{S(\Lambda_{s_c})}=\infty,$ then it must hold that $M[u_c]=1$ and  $E[u_c]=(ME)_c$,
which will complete the proof. Thus, it suffices to establish this claim.
We argue  by contradiction to suppose otherwise that
$$A\equiv\|FNLS(t-t_{n}^{1})\tilde{\psi}^1\|_{S(\Lambda_{s_c})}
=\|FNLS(t)\tilde{\psi}^1\|_{S(\Lambda_{s_c})}
=\|u_c\|_{S(\Lambda_{s_c})}<\infty.$$
By the long-time perturbation theory Proposition \ref{properturb}, we obtain $\epsilon_0=\epsilon_0(A).$
Taking $M$ as sufficiently large and $n_2(M)$ as large enough that for $n>n_2$,
it holds that $\|W_{n}^{M}\|_{S(\Lambda_{s_c})}\leq \epsilon_0.$
Similar to the proof in the  first case, Proposition \ref{properturb}
implies that there exists a large $n$ such that $\|u_c\|_{S(\Lambda_{s_c})}<\infty,$
which is a contradiction.

\end{proof}

\begin{proposition}\label{procompact}
(Precompactness of the flow of the critical solution)
Let $u_c$ be as  in Proposition \ref{critical}; then, if $\|u_c\|_{S([0,+\infty);\Lambda_{s_c})}=\infty$,
$$\{u_c(\cdot,t)| ~t\in[0,+\infty)\}\subset H^s$$
 is precompact in $H^s$.  A corresponding conclusion is reached if
 $\|u_c\|_{S((-\infty,0];\Lambda_{s_c})}=\infty$.

\end{proposition}

\begin{proof}

We will argue by contradiction and write $u=u_c$ for short.
Otherwise, we will obtain
an $\eta>0$ and a sequence $t_n\rightarrow+\infty$ such that for all $n\neq n'$,
\begin{align}\label{eta}
\|u(\cdot,t_n)-u(\cdot,t_{n'})\|_{H^s}\geq\eta.
\end{align}
We take $\phi_n=u(t_n)$ in the profile expansion lemma \ref{lpd}
to obtain the profiles $\psi^j$ and a remainder $W_n^M$ such that
$u(t_n)=\sum_{j=1}^M U(-t_n^j)\psi^j+W_n^M$
with $|t_{n}^{j}-t_{n}^{k}|\rightarrow+\infty$ as $n\rightarrow+\infty$ for any $j\neq k$.
Then, Lemma \ref{energy  expansion} gives
$$\sum_{j=1}^M\lim_{n\rightarrow+\infty}E[U(-t_n^j)\psi^j]+\lim_{n\rightarrow+\infty}E[W_n^M]=E[u(t_n)]=(ME)_c.$$
Similar to the proof of Lemma \ref{lemcompare}, we know that each energy is non-negative, and thus, for any $j$,
$$\lim_{n\rightarrow+\infty}E[U(-t_n^j)\psi^j]\leq(ME)_c.$$
Moreover, by \eqref{hsexpansion}, we have
$$\sum_{j=1}^MM[\psi^j]
+\lim_{n\rightarrow+\infty}M[W_n^M]=\lim_{n\rightarrow+\infty}M[u(t_n)]=1.$$

If more than one $\psi^j\neq0,$ following the proof in Proposition \ref{critical}, we can show that
this case will contradict  the definition of the critical solution $u=u_c$.
Thus, we will address the case in which only $\psi^1\neq0$ and $\psi^j=0$ for all $j>1,$
and thus, \begin{align}\label{513}
u(t_n)= U(-t_{n}^{1})\psi^1+W_n^M.
\end{align}
In addition, as in the proof of Proposition \ref{critical}, we find that
$M[\psi^1]=1$, $\lim\limits_{n\rightarrow+\infty}E[U(-t_{n}^{1})\psi^1]=(ME)_c,$
$\lim\limits_{n\rightarrow+\infty}M[W_n^M]=0$ and
$\lim\limits_{n\rightarrow+\infty}E[W_n^M]=0.$
Thus, by Lemma \ref{lemcompare}, we obtain
\begin{align}\label{514}
\lim_{n\rightarrow+\infty}\|W_n^M\|_{H^s}=0.
\end{align}

We claim now that $t_{n}^{1}$ converges to some finite $t^1$ up to a subsequence. Note that if this holds, because
$ U(-t_{n}^{1})\psi^1\rightarrow  e^{-it^{1}\Delta}\psi^1$ in $H^s$ and by \eqref{513},
\eqref{514} implies that $u(t_n)$ converges in $H^s$,which contradicts \eqref{eta}; we thus conclude our proof.

Now, we show the above claim by contradiction. Suppose that $t_{n}^{1}\rightarrow-\infty.$
Then, $$\|U(t)u(t_n)\|_{S(\Lambda_{s_c};[0,\infty))}\leq
\|U(t-t_{n}^{1})\psi^1\|_{S(\Lambda_{s_c};[0,\infty))}
+\|U(t)W_n^M\|_{S(\Lambda_{s_c};[0,\infty))}. $$
Because
$$\lim_{n\rightarrow+\infty}\|U(t-t_{n}^{1})\psi^1\|_{S(\Lambda_{s_c};[0,\infty))}
=\lim_{n\rightarrow+\infty}\|U(t)\psi^1\|_{S(\Lambda_{s_c};[-t_{n}^{1},\infty))}=0$$
and $\|U(t)W_n^M\|_{S(\Lambda_{s_c})}\leq \frac{1}{2}\delta_{sd}, $
by taking $n$ as sufficiently large,
we obtain a contradiction to the small data scattering theory.
If other $t_{n}^{1}\rightarrow+\infty,$
we  similarly obtain
$$\|U(t)u(t_n)\|_{S(\Lambda_{s_c};(-\infty,0])}\leq\frac{1}{2}\delta_{sd}. $$
Thus,  the small data scattering theory (Proposition \ref{sd})
shows that $$\|u\|_{S(\Lambda_{s_c};(-\infty,t_n])}\leq \delta_{sd}.$$
Because $t_n\rightarrow+\infty$ by the assumption in the beginning of our proof,
sending $n\rightarrow+\infty$, we obtain $\|u\|_{S(\Lambda_{s_c};(-\infty,+\infty))}\leq \delta_{sd},$
which is a contradiction.

\end{proof}

\begin{corollary}\label{compact-localization}
Let u be a solution to \eqref{eq1} such that $\mathcal{K}^+=\{u(\cdot,t)| ~t\in[0,+\infty)\}$
is precompact in $H_r^s$. Then, for each $\epsilon>0,$
there exists $R>0$ such that
$$\int_{|x|>R}|D^s u(x,t)|^2+|u(x,t)|^2+(\frac1{|\cdot|^\gamma}\ast|u|^2)|u|^2(x,t)dx\leq\epsilon.$$
\end{corollary}
\begin{proof}
If not, for any $R>0$, there exists $\epsilon_0>0$ and a sequence $t_n$ such that
$$\int_{|x|>R}|D^s u(x,t_n)|^2+|u(x,t_n)|^2+(\frac1{|\cdot|^\gamma}\ast|u|^2)|u|^2(x,t_n)dx\geq\epsilon_0.$$
By the precompactness of $\mathcal{K}^+$, there exists $\phi\in H^s$ such that, up to a subsequence of $t_n$,
we have $u(\cdot,t_n)\rightarrow\phi$ in $H^s$. Thus, for any $R>0$, we obtain
$$\int_{|x|>R}|D^s \phi(x)|^2+|\phi(x)|^2+(\frac1{|\cdot|^\gamma}\ast|\phi|^2)|\phi|^2(x)dx\geq\epsilon_0,$$
from which we can easily obtain a contradiction because  $\phi\in H^s$ and
$V(\phi)\leq c\|\phi\|^4_{H^s}$ by the Hardy-Littlewood-Sobolev inequality.

\end{proof}

\section{Rigidity theorem }

In this section, we will prove the following Liouville-type theorem.

\begin{theorem}\label{rigidity} Let $N\geq 2$ and $2s<\gamma<\min\{N,4s\}$.
Suppose that $u_0\in H^s$ is radial and that $u_0\in K_1$, i.e.,
\begin{equation}\label{6.1}
M[u_{0}]^{\frac{s-s_c}{s_c}}E[u_{0}]<M[Q]^{\frac{s-s_c}{s_c}}E[Q],
\end{equation}
 and
\begin{equation}\label{6.2}
M[u_{0}]^{\frac{s-s_c}{s_c}}\|u_{0}\|^2_{\dot{H}^s}<M[Q]^{\frac{s-s_c}{s_c}}\| Q\|^2_{\dot{H}^s}.
\end{equation}
Let $u$ be the global solution of \eqref{eq1} with initial data $u_0$, and it holds that
$\mathcal{K}^+=\{u(\cdot,t)| ~t\in[0,+\infty)\}$ is precompact in $H^s$.
Then, $u_0=0$.
The same conclusion holds if
$\mathcal{K}^-=\{u(\cdot,t):t\in(-\infty,0]\}$ is precompact in $H^s$.
\end{theorem}

 Before proving the rigidity theorem, we follow the same idea of \cite{b-h-ljfa16} to introduce
the localized virial estimate for the radial solutions of \eqref{eq1}.

For $u\in H^s$ with $s\geq\frac12$, we need the auxiliary function $u_m=u_m(t,x)$, defined as
\begin{align}\label{um}
u_m:=c_s\frac1{-\Delta+m}u(t)=c_s\mathcal F^{-1}\frac{\widehat u(t,\xi)}{|\xi|^2+m}
\end{align}
with $c_s=\sqrt{\frac{sin\pi s}{\pi}}$, turns out to be a convenient normalization factor.
By Balakrishnan's formula in semi-group theory used in \cite{b-h-ljfa16}, for any  $u\in H^s$, we have the identity
\begin{align}\label{snorm}
\int_0^\infty m^s\int_{\mathbb R^N}|\nabla u_m|^2dxdm=s\|(-\Delta)^{\frac s2}u\|_2^2.
\end{align}
We obtain a counterpart of Corollary \ref{compact-localization}.
\begin{corollary}\label{corlocal}
Let $u=u(t)$ be a solution to \eqref{eq1} such that $\mathcal{K}^+=\{u(\cdot,t)| ~t\in[0,+\infty)\}$
is precompact in $H_r^s$. Then, for each $\epsilon>0,$
there exists $R>0$ such that
$$\int_0^\infty m^s\int_{|x|>R}|\nabla u_m|^2dxdm+\int_{|x|>R}|u(x,t)|^2+(\frac1{|\cdot|^\gamma}\ast|u|^2)|u|^2(x,t)dx\leq\epsilon.$$
\end{corollary}

{\bf The Proof of Theorem \ref{rigidity}.} It suffices to address the $\mathcal{K}^+$ case, since the $\mathcal{K}^-$ case follows similarly.
For some given real-valued function  $\varphi\in C_c^\infty$, which is radial, with
$$\varphi(x)=\left\{\begin{aligned}
&|x|^2 \ \ &for\ \  |x|\leq1\\
&0\ \ &for\ \ |x|\geq2.\end{aligned}\right.$$
For $R>0$, define the localized virial of $u\in H^s$ to be the quantity given by
\begin{align*}
\mathcal M_{R}(t):=2Im\int_{\mathbb R^N}\bar u(t,x)R\nabla \varphi(\frac xR)\cdot\nabla u(t,x)dx.
\end{align*}
Following the method used in \cite{b-h-ljfa16}, we have the identity
\begin{align*}
\mathcal M'_{R}(t)=\int_0^\infty m^s\int_{\mathbb R^N}
\left(4\overline{\partial_k u_m}(\partial^2_{kl}\varphi(\frac xR))\partial_lu_m-(\frac1{R^2}\Delta^2\varphi(\frac xR))|u_m|^2
\right)dxdm+I,
\end{align*}
where
\begin{align*}
I=&2R\int_{\mathbb R^N}\nabla\phi(\frac xR)(\nabla(\frac1{|\cdot|^\gamma})\ast|u|^2)|u|^2dx\\
=&-\gamma R\int\int(\nabla\phi(\frac xR)-\nabla\phi(\frac xR))\cdot\frac{x-y}{|x-y|^{\gamma+2}}|u(x)|^2|u(y)|^2
\end{align*}

By the definition of $\varphi,$ we have
\begin{align}\label{MR}
\mathcal M'_R(t)=&8\int_0^\infty m^s\int_{|x|\leq R}|\nabla u_m|^2dx+
4\int_0^\infty m^s\int_{R<|x|<2R}\partial^2_r\varphi\left(\frac{x}{R}\right)|\nabla u_m|^2dxdm
\\ \nonumber
&-\frac{1}{R^2}\int_0^\infty m^s\int_{|x|>R}\Delta^2\varphi\left(\frac{x}{R}\right)|u_m|^2dxdm+I.
\end{align}
We rewrite $I$ as
\begin{align*}
I&=-\gamma R\int\int\left(\nabla\varphi\left(\frac{x}{R}\right)-\nabla\varphi\left(\frac{y}{R}\right)\right)\cdot
\frac{x-y}{|x-y|^{\gamma+2}}|u(x)|^2|u(y)|^2dxdy\\
&=-2\gamma\int\int_{\{|x|\leq R,|y|\leq R\}}\frac{|u(x)|^2|u(y)|^2}{|x-y|^{\gamma}}dxdy\\
&-\gamma R\left[\int\int_\Omega+\int\int_\Lambda\right]\left(\nabla\varphi\left(\frac{x}{R}\right)
-\nabla\varphi\left(\frac{y}{R}\right)\right)\frac{x-y}{|x-y|^{\gamma+2}}|u(x)|^2|u(y)|^2dxdy,
\end{align*}
where $$\Omega=\{(x,y)\in \mathbb R^N\times \mathbb R^N :R<|x|<2R\}\bigcup\{(x,y)\in \mathbb R^N\times \mathbb R^N :R<|y|<2R\}$$ and
$$\Lambda=\{(x,y)\in \mathbb R^N\times \mathbb R^N :|x|>2R,|y|<R\}\bigcup\{(x,y)\in \mathbb R^N\times \mathbb R^N :|x|<R,|y|>2R\}.$$
Then, by the properties of $\varphi,$ we   estimate $I$ as
\begin{align*}
I=&-2\gamma\int\int\frac{|u(x)|^2|u(y)|^2}{|x-y|^{\gamma}}dxdy\\
&+O\left(\int\int_{\{|x|\geq R\}}\frac{|u(x)|^2|u(y)|^2}{|x-y|^{\gamma}}dxdy
+\int\int_{\{|y|\geq R\}}\frac{|u(x)|^2|u(y)|^2}{|x-y|^{\gamma}}dxdy\right)\\
&+O\left(R\int\int_{\{|x|>R,|x-y|>\frac{R}{2}\}}\left(\nabla\varphi\left(\frac{x}{R}\right)-\nabla\varphi\left(\frac{y}{R}\right)\right)\frac{x-y}{|x-y|^{\gamma+2}}|u(x)|^2|u(y)|^2dxdy\right)\\
&+O\left(R\int\int_{\{|x|>R,|x-y|<\frac{R}{2}\}}\left(\nabla\varphi\left(\frac{x}{R}\right)-\nabla\varphi\left(\frac{y}{R}\right)\right)\frac{x-y}{|x-y|^{\gamma+2}}|u(x)|^2|u(y)|^2dxdy\right)\\
=&-2\gamma\int\int\frac{|u(x)|^2|u(y)|^2}{|x-y|^{\gamma}}dxdy+O\left(\int_{|x|>R}(\frac1{|\cdot|^\gamma}\ast|u|^2)|u|^2dx
\right).
\end{align*}

From \eqref{MR}, we obtain
\begin{align*}
\mathcal M'_R(t)=&8\int_0^\infty m^s\int_{|x|\leq R}|\nabla u_m|^2dx+
4\int_0^\infty m^s\int_{R<|x|<2R}\partial^2_r\varphi\left(\frac{x}{R}\right)|\nabla u_m|^2dxdm
\\
&-\frac{1}{R^2}\int_0^\infty m^s\int_{|x|>R}\Delta^2\varphi\left(\frac{x}{R}\right)|u_m|^2dxdm+I\\
\geq&\left(8\int_0^\infty m^s\int_{\mathbb R^N}|\nabla u_m|^2dx-2\gamma V(u)\right)+A_R(u)\\
=&2\gamma\left(\frac{4s}\gamma\|D^s u\|_2^2-V(u)
\right)+A_R(u).
\end{align*}
where by Corollary \ref{corlocal},
\begin{align}\label{AR}
A_R(u(t))&\leq c\left(\|D^s u\|^2_{L^2(|x|>R)}+\frac{1}{R^2}\|u\|^2_{L^2(|x|>R)}+
\int_{|x|>R}(\frac1{|\cdot|^\gamma}\ast|u|^2)|u|^2dx\right)\\ \nonumber
&\rightarrow0,\ \ {\rm as}\ \ R\rightarrow+\infty.
\end{align}

Let a positive constant $\delta\in(0,1)$ be such that
$E[u_0]<(1-\delta)E[Q]M[Q]^{\frac{s-s_c}{s_c}}$.
 It follows from
Lemma \ref{lemlowerbound} and Lemma \ref{lemcompare}  that
$$\frac{4s}\gamma\|D^s u\|_2^2-V(u)\geq C_\delta\|D^s u_0\|_2^2,$$
which gives that for large $R$,
\begin{align}\label{lowerM'}
\mathcal M'_R(t)\geq&C_\delta\|D^s u_0\|_2^2.
\end{align}
Integrating \eqref{lowerM'} over $[0,t]$, we obtain
$$|\mathcal M_R(t)-\mathcal M_R(0)|\geq C_\delta t\|D^s u_0\|_2^2$$
On the other hand, by \cite{b-h-ljfa16}, we should
have
$$|\mathcal M_R(t)-\mathcal M_R(0)|\leq C_R(\|u\|^2_{H^{\frac12}}+\|u_0\|^2_{H^{\frac12}})
\leq C_R(\|u\|^2_{H^{s}}+\|u_0\|^2_{H^{s}})\leq C_R\|Q\|^2_{H^{s}},$$
which is a contradiction for large $t$ unless $u_0=0$.
\ \  \ \ \ \ \ \ \ \ \ \ \  \ \ \ \ \ \ \ \ \ \ \  \ \ \ \ \ \ \ \ \ \ \  \ \ \ \ \ \ \ \ \ \ \  \ \ \ \ \ \ \ \ \ \ \ $\Box$

\vspace{0.5cm}

Now, we can finish the proof of Theorem 1.1.\\
{\bf The  Proof of Theorem 1.1.}

Note that by Proposition \ref{procompact},
the critical solution $u_c$ constructed in Section 4 satisfies the hypotheses in Theorem \ref{rigidity}.
Therefore,
to complete the proof of Theorem \ref{th1}, we should apply Theorem \ref{rigidity} to
$u_c$ and find that $u_{c,0}=0$, which contradicts the fact that $\|u_c\|_{S(\Lambda_{s_c})}=+\infty.$
This contradiction  shows that $SC(u_0)$ holds. Thus, by Proposition \ref{h1scattering},  we have shown that $H^s$ scattering holds.
\ \  \ \ \ \ \ \ \ \ \ \ \  \ \ \ \ \ \ \ \ \ \ \  \ \ \ \ \ \ \ \ \ \ \  \ \ \ \ \ \ \ \ \ \ \  \ \ \ \ \ \ \ \ \ \ \ $\Box$

\vspace{0.5cm}

{\bf Acknowledgments }

This work  was  supported by the National Natural Science Foundation of China (No. 11501395,  11301564, and 11371267) and  the Excellent Youth Foundation of Sichuan  Scientific Committee grant No. 2014JQ0039  in China.

\vspace{0.5cm}

%{\bf Acknowledgements:} D.Cao is grateful to L.Caffarelli forbringing the problem studied in this paper to his attention. Both
%D.Cao and S.Peng were supported by the Key Project ofNSFC(No.10631030). D.Cao was also supported partially by Science
%Fund for Creative Research Groups of Natural Science Foundation ofChina (No.10721101). S.Peng was also supported by the Program for
%New Century Excellent Talents in University (No.07-0350). S.Yan was partiallysupported by ARC in Australia.

\end{document}